%% file: main.tex
\documentclass{article}
\input{header}

\usepackage[accepted]{icml2023}
\graphicspath{{figures/}}

\icmltitlerunning{Slowly Time-Varying Networks}

\begin{document}
	\twocolumn[
	\icmltitle{Is Consensus Acceleration Possible in Decentralized Optimization over Slowly Time-Varying Networks?}

	\icmlsetsymbol{equal}{*}

	\begin{icmlauthorlist}
		\icmlauthor{Dmitriy Metelev}{MIPT}
		\icmlauthor{Alexander Rogozin}{MIPT,HSE}
		\icmlauthor{Dmitry Kovalev}{UCL}
		\icmlauthor{Alexander Gasnikov}{MIPT,ISP,IITP}
	\end{icmlauthorlist}

	\icmlaffiliation{MIPT}{Moscow Institute of Physics and Technology, Moscow, Russia}
	\icmlaffiliation{UCL}{Universit\'e Catholique de Louvain, Ottignies-Louvain-la-Neuve, Belgium}
	\icmlaffiliation{HSE}{HSE University, Moscow, Russia}
	\icmlaffiliation{ISP}{ISP RAS Research Center for Trusted Artificial Intelligence, Moscow, Russia}
	\icmlaffiliation{IITP}{Institute of Information Transmission Problems, Moscow, Russia}

	\icmlcorrespondingauthor{Alexander Rogozin}{aleksandr.rogozin@phystech.edu}

	\icmlkeywords{convex optimization, decentralized optimization, time-varying network}

	\vskip 0.3in
	]

	\printAffiliationsAndNotice{}  


	\begin{abstract}
		We consider decentralized optimization problems where one aims to minimize a sum of convex smooth objective functions distributed between nodes in the network. The links in the network can change from time to time. For the setting when the amount of changes is arbitrary, lower complexity bounds and corresponding optimal algorithms are known, and the consensus acceleration is not possible. However, in practice the magnitude of network changes may be limited. We derive lower communication complexity bounds for several regimes of velocity of networks changes. Moreover, we show how to obtain accelerated communication rates for a certain class of time-varying graphs using a specific consensus algorithm.

	\end{abstract}

	\section{Introduction}
	In this paper we consider a decentralized optimization problem 
	\begin{equation}\label{main_problem}
	    \min_{x
	    \in \R^m}~ f(x) = \frac{1}{n}\sum_{i=1}^n{f_i(x)},
	\end{equation}
	where each function $f_i$ is convex, has a Lipschitz gradient and is stored at a separate computational node. Nodes are connected by a communication network (that may change over time). Each node is an independent computational agent that can perform local computations based only on the information in its local memory. At each communication step, nodes can only exchange information with their neighbours.
	
	Sum-type problems of type \eqref{main_problem} have applications in practical scenarios where centralized coordination is not possible. Communication constraints may appear due to large amounts of data or due to privacy constraints \cite{konevcny2016federated} and are determined by the structure of the network. Decentralized optimization is widely used in distributed machine learning \cite{rabbat2004distributed,forero2010consensus,nedic2020distributed,gorbunov2022recent}, distributed control \cite{ram2009distributed,gan2012optimal} and distributed sensing \cite{bazerque2009distributed}.
	

	\subsection{Time-Varying Networks}
	
	We study the setting when the network is \textit{time-varying}. That means that the links between the nodes may appear and disappear from time to time. In practice, the changes in the links may occur due to loss of wireless connection between the agents or other technical malfunctions. Note that while the set of edges may change, the set of vertices stays the same.

	\subsection{Related work}
	
	In this paper, we assume the objective $f(x)$ in \eqref{main_problem} to be $L$-smooth and $\mu$-strongly convex. Complexity bounds for decentralized optimization include two quantities: objective condition number $\kappa_g = L/\mu$ and network condition number $\chi$. In case of the time-varying network, $\chi$ denotes the worst-case condition number over time steps. 

	Lower complexity bounds for optimization over static graphs were proposed in \cite{scaman2017optimal}. The lower communication complexity bound is $\Omega(\kappa_g^{1/2} \chi^{1/2} \log(1/\eps))$. 
	The corresponding optimal algorithms are MSDA \cite{scaman2017optimal} (using dual oracle) and OPAPC \cite{kovalev2020optimal} (using primal oracle). 
	
	For time-varying networks, the lower communication complexity bound is $\Omega(\kappa_g^{1/2} \chi \log(1/\eps))$ \cite{kovalev2021lower}. The corresponding optimal algorithms with primal oracle are ADOM+ \cite{kovalev2021lower} and Acc-GT \cite{li2021accelerated} with multi-step communication. An optimal dual algorithm is ADOM \cite{kovalev2021adom}. Prior to optimal algorithms, several non-accelerated schemes like DIGing \cite{nedic2017achieving} and sub-optimal methods with additional logarithmic factor, i.e. APM-C \cite{li2020decentralized,dvinskikh2019decentralized,rogozin2021towards} and Mudag \cite{ye2020multi} were proposed.
	
	Optimal algorithms both for static and time-varying scenarios use a multi-step consensus scheme. In the time-static case, the communication matrix is replaced by a Chebyshev polynomial of it \cite{scaman2017optimal}. The degree of polynomial is $\lceil\chi^{1/2}\rceil$ and its condition number is $O(1)$. In the time-varying case, after each oracle call multiplication is performed by $\chi$ matrices in a row instead of only one matrix \cite{kovalev2021lower}.

	\subsection{Contributions}

	The case when arbitrarily many edges can change at each time step is well-studied. However, we think that such a setting is not realistic and in practice the magnitude of graph changes may be limited. We investigate several types of such restrictions and derive lower complexity bounds for each case. This constitutes the first part of our work (see Table \ref{tab:lower_bounds}). We show that it is sufficient to change a polynomial number of vertices (i.e. $O(n^\alpha)$ for some $\alpha > 0$) at each iteration in order to slow consensus speed down to factor $\chi$. Moreover, if a logarithmic number of edges is changed (i.e. $O(\log n)$), the consensus is slowed down to $\chi/\log\chi$. Finally, our results suggest that a partial consensus acceleration (i.e. dependency on $\chi$ in power between $1/2$ and $1$) is possible if the number of changes is bounded by a constant.

	\begin{table}[ht]
		\caption{Known lower communication complexity bounds for decentralized optimization and our results. Here $\alpha>0$ is a scalar and $d\in\N$ is a constant. The complexity depends on the maximum number of changes in links allowed at each iteration.}
		\label{tab:lower_bounds}
		{\footnotesize
			\begin{tabular}{|b{1.38cm}|b{2.95cm}|>{\centering\arraybackslash}b{2.9cm}|}
				\hline
				Number of & \parbox[t]{20mm}{\multirow{2}{*}{Lower bound}} & \parbox[t]{20mm}{\multirow{2}{*}{Reference}} \\
				changes & & \\
				\hline
				no changes & \hspace{-0.0cm}$\Omega\cbraces{\chi^{1/2}\kappa_g^{1/2}\log\frac{1}{\eps}}$ & \cite{scaman2017optimal} \\ \hline
				$O(n)$ & \hspace{-0.0cm}$\Omega\cbraces{\chi\kappa_g^{1/2}\log\frac{1}{\eps}}$ & \cite{kovalev2021lower} \\ \hline
				$O(n^\alpha)$ & \hspace{-0.0cm}$\Omega\cbraces{\chi\kappa_g^{1/2}\log\frac{1}{\eps}}$ & This paper, Th.~\ref{theorem:poly} \\ \hline
				$O(\log n)$ & \hspace{-0.0cm}$\Omega\cbraces{\frac{\chi}{\log\chi}\kappa_g^{1/2}\log\frac{1}{\eps}}$ & This paper, Th.~\ref{theorem:log} \\ \hline
				$12(d - 1)$ & \hspace{-0.0cm}$\Omega\cbraces{\chi^{d/(d+1)}\kappa_g^{1/2}\log\frac{1}{\eps}}$ & This paper, Th.~\ref{theorem:const} \\ \hline
			\end{tabular}
		}
	\end{table}

	In the setting where a constant number of edges changes at each iteration, our results allow to establish the known lower bounds for static graphs and time-varying graphs with arbitrary changes. The corresponding results are presented in the last line of Table~\ref{tab:lower_bounds}. Putting $d = 1$ leads to the static case and the lower bound coincides with the one in \cite{scaman2017optimal}. In the opposite case, taking $d\to\infty$ leads to the scenario with arbitrary changes, and the corresponding lower bound approaches the results in \cite{kovalev2021lower}. In other words, our results suggest an interpolation between two edge cases: static graphs and time-varying graphs with arbitrary changes.

	In the second part of our paper, we address a multi-step consensus technique for time-varying graphs. More precisely, we apply Nesterov acceleration technique to time-varying consensus. The acceleration is attained under an additional assumption: we assume that all graphs have a common connected subgraph that we call a \textit{skeleton}. On the one hand, this assumption is more strict then requiring the network to stay connected all the time. On the other hand, we think that such an assumption may be realistic in practical scenarios.

	The consensus procedure for graphs with connected skeleton is slightly modified: the two nodes stop communicating to each other if the connection between them has been lost at least once. In other words, the active links in the communication graph are not restored after they have failed at least once, and therefore the network is "monotonically decreasing".

	Summing up, this paper makes a step in the direction of optimization over special classes of time-varying networks. Our lower bounds show that acceleration communication protocol is hard to be designed even over slowly time-varying graphs. On the other hand, we show a specific class of networks over which accelerated consensus is reachable.
	
	The paper is organized as follows. In Section~\ref{sec:definitions_and_assumptions} we introduce notation, definitions and assumptions. In Section~\ref{sec:lower_bounds}, we present our main results on lower bounds. After that, in Section~\ref{sec:accelerated_gossip_tw} we describe the accelerated gossip protocol over time-varying networks with connected skeleton.

    \section{Definitions and Assumptions}\label{sec:definitions_and_assumptions}
    
	We denote Kronecker product by $\otimes$. The nullspace of matrix $\mA$ is denoted $\ker\mA$ and the range of $\mA$ is denoted $\range\mA$. Moreover, if $\mA$ is symmetric and positive semi-definite, we denote its largest eigenvalue $\lambda_{\max}(\mA)$, its minimal nonzero eigenvalue $\lambda_{\min}^+(\mA)$ and its condition number $\chi(\mA) = \lambda_{\max}(\mA) / \lambda_{\min}^+(\mA)$.  For vectors $x_1, \ldots, x_n\in\R^d$, we introduce a column stacked vector $\bx = \col[x_1, \ldots, x_n] = (x_1^\top \ldots x_n^\top)^\top\in\R^{nd}$. We also denote $\N = \{1, 2, \ldots\}$ to be the set of positive integers.
    
    \subsection{Objective Functions}
    
	Let $H$ be an arbitrary Hilbert space, let $\norm{\cdot}$ be the norm on $H$ and let $\norm{\cdot}_*$ denote the conjugate norm. 
	\begin{definition}
	    Function $h(x):~ H\to\R$ is called $L$-smooth if for any $x, y\in H$ it holds
	    \begin{align*}
	        \norm{\nabla h(y) - \nabla h(x)}_*\leq L\norm{y - x}.
	    \end{align*}
	\end{definition}
	\begin{definition}
	    Function $h(x):~ H\to\R$ is called $\mu$-strongly convex if for any $x, y\in H$ it holds
	    \begin{align*}
	        h(y)\geq h(x) + \angles{\nabla h(x), y - x} + \frac{\mu}{2}\norm{y - x}^2.
	    \end{align*}
	\end{definition}
	Throughout the paper we only work with $H = l_2$. Although the lower bounds are meat for optimization in a finite dimension, they are derived for $l_2$, as it is typically done in optimization of strongly convex smooth functions \cite{nesterov2004introduction}.

	\subsection{Decentralized Communication}
	We assume that distributed communication is performed via a series of communication rounds. In each of the rounds, the nodes interact through a network represented by an undirected communication graph $\cG_k = (\cV, \cE_k)$, where $k\in\N$ is the current iteration number. The nodes can only communicate to their immediate neighbors in the corresponding network.
	
	Note that the set of nodes $\cV$ does not change over time. Throughout the paper we only consider the case when all graphs $\cG_k$ are connected.
	
	In the literature, analysis of optimization algorithms in the decentralized setting as well as the lower bounds are usually based on the condition number of gossip matrices.
	\begin{assumption}\label{assum:gossip_matrix}
	    Matrix $W_k\in \R^{n\times n}$ is called a gossip matrix of undirected graph $\cG_k=\(\cV, \cE_k\)$ if the following properties are satisfied:
	    \begin{enumerate}
	        \item $W_k$ is symmetric positive semi-definite,
	        \item $[W_k]_{i, j} = 0$ if $i\neq j$ and $(i, j)\not \in \cE_k$,
	        \item $\ker W_k = \{(x_1, \dots, x_n)\in\R^n:x_1=\dots=x_n\}$.
	    \end{enumerate}
	\end{assumption}

	Given a gossip matrix $W_k$, we define $\mW_k = W_k\otimes \mI_d$. Then $\mW_k$ is also symmetric and positive semi-definite, multiplication by $\mW_k$ represents one communication round and $\ker\mW_k = \{\bx = \col[x_1, \ldots, x_n]\in \R^{nd}:~ x_1 + \ldots+ x_n = 0\}$.
	
	For a given gossip matrix $W$, introduce its condition number $\chi(W) = \lambda_{\max}(W) / \lambda_{\min}^+(W)$.
	
	A common example of a communication matrix is the graph Laplacian $\Lap(\cG_k) = D(\cG_k) - A(\cG_k)$, where $A(\cG_k)$ denotes the adjacency matrix of $\cG_k$ and $D(\cG_k) = \text{diag}(\sum_i A_{ij})$ is a diagonal matrix with degrees of the nodes at diagonal. Laplacian matrix $\Lap(\cG_k)$ satisfies Assumption~\ref{assum:gossip_matrix}.
	
	Further in the paper we will use only Laplacian matrices, therefore by slight abuse of notation we denote $\chi(G) = \chi(\Lap(G))$.

    \subsection{Decentralized Problem and Slowly-Changing Setup}
	The main results of our work are lower bounds for decentralized optimization problems. We formalize the definition of decentralized problem and its characteristics.
    \begin{definition}
        Let us define $\DP$ (decentralized time-varying problem) as a pair $(\{\cG_k\}_{k=1}^\infty, \{f_i\}_{i=1}^n)$. Firstly, $\DP$ includes a sequence of undirected connected graphs $\{\cG_k\}_{k=1}^{\infty}$ with a common set of vertices $\cV = \{1, 2, \ldots, n\}$ and edge sets $\{\cE_k\}_{k=1}^\infty$. Let $n(\DP) = n$, $\chi(\DP) = \sup_{k\in \N} \chi(\cG_k)$. Secondly, $\DP$ includes a set of objective functions $\{f_i\}_{i=1}^n$. We refer to $f(x) = \frac{1}{n}\sum_{i=1}^n f_i(x)$ as a global function.
    \end{definition}

    \begin{definition}
        Decentralized time-varying problem $\DP$ is called $L$-smooth if global function $f$ is $L$-smooth. 
    \end{definition}

    \begin{definition}
        Decentralized time-varying problem $\DP$ is called convex ($\mu$-strongly convex)  if global function $f$ is convex ($\mu$-strongly convex).
    \end{definition}
%
    
    \begin{definition}
        Let $\Delta(\DP)$ denote the maximum amount of edges that change between consequent communication rounds. Particularly, $\Delta(\DP) = \max_{k\in \N}\left\{\sum_{i<j}\mathbb{I}_{ij}(\cE_k, \cE_{k+1})\right\}$, where
        $$\mathbb{I}_{ij}(\cE,\cE')=\begin{cases}
            1, & \text{if } (i, j) \in (\cE\setminus \cE')\cup(\cE'\setminus \cE), \\
            0, & \text{otherwise}.
        \end{cases}$$
    \end{definition}
    
	The introduced quantity $\Delta(\DP)$ expresses the maximum change in edges between two consequent time steps. Later in the paper we show that the value of $\Delta(\DP)$ regulates the magnitude of network changes and determines the dependence of the lower bounds on condition number $\chi$.
	
	\textbf{Example 1}. Consider a static graph and denote its corresponding problem $\DP_{static}$. All of the graphs in $\DP_{static}$ are the same, therefore, we have $\Delta(\DP_{static}) = 0$.
	
	\textbf{Example 2}. Consider the example used in lower bounds in \cite{kovalev2021lower}. The authors proposed a star graph which center changes at each iteration (denote the corresponding problem $\DP_{star}$). In such a setting, at every iteration every edge in the graph changes. We have $\Delta(\DP_{star}) = 2(n - 1)$.

	\begin{figure}[ht]
		\centering
		\includegraphics[width=0.48\textwidth]{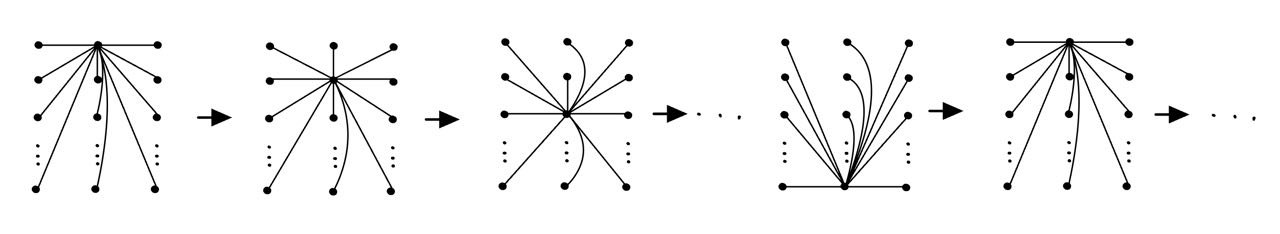}
		\caption{Example of a time-varying network with all edges changing at each iteration \cite{kovalev2021lower}}
		\label{fig:example_adom_plus}
	\end{figure}

    \section{Lower Bounds}\label{sec:lower_bounds}

    This section presents three results corresponding to different constraints on the rate of change of the communication graph: a polynomial constraint on the change of edges per iteration, a logarithmic constraint, and a constant constraint. We show that each of these regimes leads to a different complexity dependency on the condition number of the gossip matrix.
    
    \subsection{First-order Decentralized Algorithms}
    
    Let us first formalize the procedure for which we derive the lower bounds. Following the definitions of \cite{kovalev2021lower} and \cite{scaman2017optimal}, we consider time steps $k\in\N$ and introduce local memory $\cH_i(k)$ for each of the agents at time step $k$. At each time step, the agents can either communicate or perform local computations. For each time step $k$, denote the last preceding communication time as $q(k)$.
    
    1. If nodes perform a local computation at step $k$, local information is updated as
    \begin{align*}
    \cH_i(k+1)\subseteq \spn\cbraces{\braces{x, \nabla f_i(x), \nabla f_i^*(x):~ x\in\cH_i(k)}}
    \end{align*}
    for all $i = 1, \ldots, n$.
    
    2. If the nodes perform a communication round at time step $k$, local information is updated as
    \begin{align*}
    \cH_i(k+1)\subseteq \spn\cbraces{\bigcup_{j\in\cN_i^{q(k)}\cup\{i\}} \cH_j(k)}
    \end{align*}
    for all $i = 1, \ldots, n$. Here $\cN_i^{q(k)}$ is a set of neighbors of agent $i$ at time step $q(k)$, i.e. at the time of last communication.
    
    
    \subsection{Main Results}
    
    The following theorem discusses the polynomial constraint on the change of edges per iteration; it turns out that such a constraint leads to the same lower bound as in the unconstrained mode studied in \cite{kovalev2021lower}.
	
	\begin{theorem}\label{theorem:poly}
		For any $L\geq\mu>0$, $L > 24$, $\alpha > 0,~ c > 0, M > 0$ there exists a constant $K(\alpha, c)>0$ and $L$-smooth $\mu$-strongly convex decentralized problem $\DP$ with $n(\DP) = n > M, \chi(\DP) = \chi > M$, $\Delta(\DP) \le cn^\alpha$, such that for any first-order decentralized algorithm for all $p\in\N$ we have
		\begin{equation*}
		\norm{x_p-x_*}^2 \ge \(1-2\sqrt{6}\sqrt{\frac{\mu}{L}}\)^{\frac{K(\alpha, c)p}{\chi}+2}\norm{x_0-x_*}^2.
		\end{equation*}
	\end{theorem}

	\begin{corollary}
		For any $L\geq\mu>0$, $L > 24$, $\alpha > 0, c > 0$ there exists $L$-smooth $\mu$-strongly convex decentralized problem $\DP$ with sufficiently large $\chi(\DP) = \chi,~ n(\DP) = n$, such that $\Delta(\DP) \le cn^\alpha$, and for any first-order decentralized algorithm the number of communication rounds to find an $\epsilon$-accurate solution of the problem \ref{main_problem} is lower bounded by
		\begin{equation*}
		\Omega\(\chi\sqrt{L/\mu}\log{\frac{1}{\epsilon}}\).
		\end{equation*}
	\end{corollary}

    The following theorem corresponds to the case where the number of edges that can change per iteration is at most logarithmic in the number of nodes.

    \begin{theorem}\label{theorem:log}
        For any $L\geq\mu>0, L>10$, $M>0$ there exists $L$-smooth $\mu$-strongly convex decentralized problem $\DP$ with $n(\DP) = n > M,~ \chi(\DP)=\chi > M$, such that $\Delta(\DP) \le 12\log_2(n)$ and for any first-order decentralized algorithm for all $p\in\N$ we have
        {\small
        \begin{equation*}
            \norm{x_p-x_*}^2 \ge \(1 - \sqrt{10}\sqrt{\frac{\mu}{L}}\)^{\frac{12\log_2(\chi/2)p}{\chi}+2}\norm{x_0-x_*}^2.
        \end{equation*}
    	}
    \end{theorem}

    \begin{corollary}
		For any $L\geq\mu>0,~ L > 10$ there exists $L$ smooth and $\mu$-strongly convex decentralized problem $\DP$ with sufficiently large $\chi(\DP) = \chi,~ n(\DP) = n$, such that $\Delta(\DP) \le 12\log_2 n$, and for any first-order decentralized algorithm the number of communication rounds to find an $\epsilon$-accurate solution of the problem \ref{main_problem} is lower bounded by
		\begin{equation*}
		\Omega\(\frac{\chi}{\log \chi}\sqrt{L/\mu}\log{\frac{1}{\epsilon}}\).
		\end{equation*}
	\end{corollary}


    As we can see, although the logarithmic constraints are tighter than the polynomial ones, the problem cannot be solved much faster. The benefit we get from logarithmic constraints is only the logarithmic factor $\log\chi$, which is typically small compared to the main term $\chi$.

    The following theorem describes the case when the constraints on changes for sequential iteration are constant.

    \begin{theorem}\label{theorem:const}
        For any $L\geq\mu> 0,~ L > 24,~ M > 0$, $d \in \N$ there exists a constant $K(d)>0$ and $L$-smooth $\mu$-strongly convex decentralized problem $\DP$ with $n(\DP) = n > M,~ \chi(\DP) = \chi > M$, $\Delta(\DP) \le 12(d-1)$, such that for any first-order decentralized algorithm for all $p\in\N$ we have
        \begin{equation*}
            \norm{x_p-x_*}^2 \ge \(1 - 2\sqrt{6}\sqrt{\frac{\mu}{L}}\)^{K(d)p\chi^{-\frac{d}{d+1}}+2}\norm{x_0-x_*}^2.
        \end{equation*}
    \end{theorem}

    \begin{corollary}
		For any $L\geq\mu> 0,~ L > 24,~ d\in \N$ there exists $L$-smooth $\mu$-strongly convex decentralized problem $\DP$ with sufficiently large $\chi(\DP) = \chi$ and $n(\DP) = n$, such that $\Delta(\DP) \le 12(d - 1)$, and for any first-order decentralized algorithm the number of communication rounds to find an $\epsilon$-accurate solution of the problem \ref{main_problem} is lower bounded by
		\begin{equation*}
		\Omega\(\chi^{\frac{d}{d+1}}\sqrt{L/\mu}\log{\frac{1}{\epsilon}}\).
		\end{equation*}
	\end{corollary}


    This result shows that even with constant constraints, the lower estimates approach the estimates without constraints. This indicates that the criterion based on the Laplacian matrix condition number is very sensitive to the degree of graph variability. 

	\subsection{Discussion}
	
	In summary, lower bounds on the number of communications have been obtained for the decentralized optimization of smooth strongly convex functions with restrictions on network change rate. In particular, three modes were considered, which differ in the speed of changing of the communication graph.
	
	Polynomial change. The first mode assumes a polynomial maximum change in the number of edges in the graph; as it turned out, such a setting gives the lower bounds in this case coincide with lower bounds when no additional conditions are imposed on the change of the graph.
	
	The second mode considers a maximum logarithmic change in the number of edges in the graph. In this case, the lower bounds turned out to be close to the estimates in the case of time-varying networks without restrictions.

	The third mode considers a constant change in the number of edges; here, for each value of the constant $d$, an individual estimate is obtained. When $d = 1$, we restore the lower bound for static graphs \cite{scaman2017optimal} and when $d\to\infty$, we approach a lower bound for time-varying networks with no restrictions on the speed of changes \cite{kovalev2021lower}. In other words, the lower bounds for the last mode can be seen as an interpolation between static and time-varying networks without restrictions.
	
	\textbf{Meaning of lower bounds}. It is worth noting that the interpretation of our lower complexity bounds is different from previous results in \cite{kovalev2021lower} and \cite{scaman2017optimal}. The mentioned papers build an example of a optimization problem for any $\chi > 0$, while our results suggest that a counterexample exists only if $\chi$ is sufficiently large. Let us illustrate how to interpret the lower bounds in the regime of polynomial change. Theorem~\ref{theorem:poly} means that 

	{
		\centering
		\textit{A decentralized optimization method that solves \textbf{any} decentralized problem with polynomial bound on changes in $O(\chi^p\sqrt{L/\mu}\log(1/\eps)),~ p < 1$ communication rounds \textbf{does not exist}.
		}
	}

	In other words, our results are asymptotic, but they are sufficient to restrict the area for future research.


\section{Accelerated Gossip for Time-Varying Graphs with Connected Skeleton}\label{sec:accelerated_gossip_tw}


\subsection{Time-Varying Graphs with Connected Skeleton}

It is known that the number of communications cannot be enhanced on the class of time-varying graphs that are allowed to change arbitrarily but stay connected at each iteration. The corresponding lower complexity bounds have been proposed in \cite{kovalev2021lower}. However, the lower bounds in \cite{kovalev2021adom} are built using a graph where $O(n)$ edges change at each time step. Namely, a "bad" graph is a star graph where the center of a star changes at each iteration (see Figure~\ref{fig:example_adom_plus}).

In practice, a situation where $O(n)$ edges change at every iteration may not always occur. The amplitude of network malfunctions may be not so large. We let all of the graphs in the sequence have a common subgraph (a \textit{skeleton}) that remains connected through time.
\begin{assumption}\label{assum:connected_skeleton}
	Graph sequence $\{\cG_k = (\cV, \cE_k)\}_{k=0}^\infty$ has a connected skeleton: there exists a connected graph $\hat\cG = (\cV, \hat\cE)$ such that for all $k = 0, 1, \ldots$ we have $\hat\cE\subseteq\cE_k$.
\end{assumption}
\begin{assumption}\label{assum:connected_skeleton_bounded_spectrum}
	For each $k = 0, 1, \ldots$ we have $\lambda_{\max}(\Lap(\cG_k)) \leq \lambda_{\max}$. Moreover, we have $\lambda_{\min}^+\leq \lambda_{\min}^+(\hat \cG)$.
\end{assumption}
Assumption~\ref{assum:connected_skeleton} is more strict then the assumption on the graph staying connected at each iteration. However, under Assumption~\ref{assum:connected_skeleton} we propose an accelerated consensus procedure over time-varying graphs.

\subsection{Accelerated Gossip with Non-Recoverable Links}

A common approach to accelerated consensus over static graphs is Chebyshev acceleration proposed in \cite{scaman2017optimal}. A gossip matrix $\mW$ with condition number $\chi(\mW)$ can be replaced by a matrix polynomial $P_K(\mW)$ of degree $K = \lceil (\chi(\mW))^{1/2} \rceil$ with condition number $\chi(P_K(\mW)) = O(1)$. The construction of $P_K(\mW)$ is based on Chebyshev polynomials of first type. Then, the condition number of communication matrix is reduced from $\chi(\mW)$ to $O(1)$ at the cost of performing $\lceil(\chi(\mW))^{1/2}\rceil$ communication rounds instead of one.

However, Chebyshev method is only known to be applied to consensus over static networks. In our work, we propose an accelerated gossip scheme over time varying graphs based on Nesterov acceleration. The acceleration is possible because of assumption on connected skeleton (Assumption~\ref{assum:connected_skeleton}) and due to a specific consensus strategy.

We use the following approach to tackle with time-varying graphs that have a connected skeleton. Let agent $i$ in the network stop exchanging information to agent $j$ once the connection between $i$ and $j$ has been lost at any communication round. In other words, if a link fails once, the communication through it is not recovered afterwards. This procedure is referred to as \textit{accelerated gossip with non-recoverable links}.
\begin{algorithm}[ht]
	\caption{Accelerated Gossip with Non-Recoverable Links}
	\label{alg:accelerated_gossip_tw}
	\begin{algorithmic}[1]
		\REQUIRE Initial guess $\bx\in\R^{nd}$, stepsizes $\eta, \beta> 0$. Set $\by^0 = \bx^0 = \bx$. 
		\STATE{Every node $i = 1, \ldots, n$ initializes set of neighbors $\cN_i = \cN_i^0$.}
		\FOR{$t = 0, 1, \ldots, T - 1$}
		\STATE{Every node $i$ does}
		\STATE{Update the set of nodes to which the node communicates: $\cN_i = \cN_i\cap\cN_i^k$}
		\STATE{$y_i^{k+1} = x_i^k - \eta (|\cN_i| x_i^k - \sum_{j\in\cN_i} x_j^k)$}
		\STATE{$x_i^{k+1} = \cbraces{1 + \beta}y_i^{k+1} - \beta y_i^k$}
		\ENDFOR
		\STATE{\algorithmicreturn~ $C_T(\bx) = \bx - \bx^T$}
	\end{algorithmic}
\end{algorithm}

Note that the output of Algorithm~\ref{alg:accelerated_gossip_tw} is $C_T(\bx)$. We claim that $C_T(\bx)$ is a linear operator that is a time-varying analogue of $P_K(\mW)\bx$.

\begin{theorem}\label{th:chebyshev_acceleration_tw}
	Let Assumptions \ref{assum:connected_skeleton} and \ref{assum:connected_skeleton_bounded_spectrum} hold. Denote $\chi = \lambda_{\max}/\lambda_{\min}^+$ and set the parameters of Algorithm~\ref{alg:accelerated_gossip_tw} to $\eta = 1/\lambda_{\max},~ \beta = (\sqrt\chi - 1)/(\sqrt\chi+ 1)$. Then operator $C_T(\bx)$ defined in Algorithm~\ref{alg:accelerated_gossip_tw} has the following properties.

	1. $C_T(\bx)$ is linear.
	
	2. $\range C_T(\bx) = \cL^\top = \{\bx\in\R^{md}:~ x_1 + \ldots + x_m = 0\}$.
	
	3. For $T = \sqrt\chi\log(4\chi)$ we have that for any $\bx\in\cL^\top$ it holds $(1 - 1/\sqrt{2})\norm{\bx}_2\leq \norm{C_T(\bx)}_2\leq (1 + 1/\sqrt{2})\norm{\bx}_2$.
\end{theorem}
%
The meaning of Theorem~\ref{th:chebyshev_acceleration_tw} is the following: for "monotone" graphs we can replace $\mW^k$ with condition number $\chi$ by $C_T(\cdot)$ with condition number $O(1)$. The payment for reduction of condition number is $\sqrt\chi\log(4\chi)$ communication rounds.

\subsection{Accelerated Gossip as Accelerated Method over Time-Varying Function}

Algorithm~\ref{alg:accelerated_gossip_tw} can be viewed as a gossip algorithm over a "monotonic" network where the edges only vanish and do not appear. Namely, introduce a sequence of graphs $\{\hat \cG_k = (\cV, \cap_{j=0}^k \cE_j)\}_{k=0}^\infty$, corresponding Laplacians $\{\hat W^k = \Lap(\hat \cG_k)\}_{k=0}^\infty$ and denote $\hat\mW^k = \hat W^k\otimes\mI_d$. Then Algorithm~\ref{alg:accelerated_gossip_tw} writes as
\begin{align}\label{eq:accelerated_gossip}
\begin{cases}
	\by^{t+1} = \bx^k - \eta\hat\mW^k\bx^k, \\
	\bx^{k+1} = (1 + \beta)\by^{k+1} - \beta\by^k.
\end{cases}
\end{align}
As can be seen from \eqref{eq:accelerated_gossip}, on each time step a multiplication by $\hat\mW^k$ is performed, which corresponds to one communication over $\hat \cG_k$. The edges in sequence $\{\hat \cG_k\}_{k=0}^\infty$ only vanish and do not appear.


Algorithm~\ref{alg:accelerated_gossip_tw} can also be interpreted as minimization of a time-varying functional with an accelerated gradient method. Consider problem
\begin{align}\label{eq:consensus_as_minimization}
	\min_{\bx\in\R^{nm}}~ h_k(\bx) = \frac{1}{2} \bx^\top \hat\mW_k \bx.
\end{align}
Algorithm~\ref{alg:accelerated_gossip_tw} is accelerated Nesterov method with step-size $\eta$ and momentum term $\beta$ applied a time-varying problem~\eqref{eq:consensus_as_minimization}. 

\begin{remark}
Returning to the case of static networks, it is worth mentioning that Chebyshev polynomials and accelerated gradient methods for quadratic minimization have a strong connection, see i.e. Chapter 2 of \cite{d2021acceleration}. In fact, Polyak's momentum can be derived through application of Chebyshev polynomials to a quadratic minimization problem.
\end{remark}

The analysis of accelerated method over a uniformly non-increasing time-varying function is based on the Lyapunov function technique.
\begin{lemma}\label{lemma:potential_decrease}
	Let Assumptions~\ref{assum:connected_skeleton} and \ref{assum:connected_skeleton_bounded_spectrum} hold. Denote $\tau = 1 / (\sqrt\chi + 1)$, $\bz_{k} = 1/\tau~\bx_{k} - (1-\tau)/\tau~\by_{k}$, $\gamma = 1 / (\sqrt\chi - 1)$ and introduce potential 
	$$
	\Psi_k = (1 + \gamma)^k \cbraces{h_k(\by^k) + \frac{\lambda_{\min}^+}{2}\norm{\bz^k - \bx^*}_2^2},
	$$
	where $\bx^*$ is a solution of \eqref{eq:consensus_as_minimization}. Then $\Psi_{k+1} - \Psi_k\leq 0$.
\end{lemma}

The proof of Lemma~\ref{lemma:potential_decrease} is based on standard analysis proposed in \cite{bansal2019potential} and on the observation that the objective function in \eqref{eq:consensus_as_minimization} is uniformly non-increasing, i.e. for any $\bx\in\R^{nd}$ and for any $k = 0, 1, \ldots$ we have $h_{k+1}(\bx)\leq h_k(\bx)$. Indeed, we have $\hat W^k = \sum_{(i,j)\in\hat \cE_k} (e_i - e_j) (e_i - e_j)^\top$, where $e_{i}$ denotes the $i$-th coordinate vector of $\R^m$. Therefore, it holds
\begin{align*}
W^k - W^{k+1} = \sum_{(i,j)\in \hat \cE_k\backslash \hat \cE_{k+1}} (e_i - e_j)(e_i - e_j)^\top\succeq 0.
\end{align*}
In other words, for any $\bx\in\R^{nm}$ it holds
\begin{align*}
	h_{k+1}(\bx) - h_k(\bx) = \frac{1}{2}\bx^\top\hat\mW^{k+1}\bx - \frac{1}{2}\bx^\top \hat\mW^k \bx\leq 0.
\end{align*}
The analysis of accelerated gradient method over time-varying uniformly non-increasing functions is presented in Appendix~\ref{sec:acc_method_over_tw_function}.

\section{Conclusion}\label{sec:conclusion}

In this paper, we study new classes of time-varying networks that may be more practical then the scenarios previously studied in the literature. We propose to look into a new direction of research -- slowly time-varying graphs. In the work, we formalize several regimes covering the velocity of graph changes and provide the corresponding lower bounds for each case. Our results outline the limits of what communication rates can be achieved over slowly time-varying graphs. Moreover, we propose a slightly modified consensus technique that leads to acceleration over time-varying networks with connected skeleton. Our technique may be seen as an analogue of Chebyshev acceleration that is used for time-static graphs.


\bibliographystyle{icml2023}
\bibliography{references}

\newpage
\appendix
\onecolumn
{\large\textbf{Supplementary material}}
\section{Proof of the Theorem~\ref{theorem:poly}}

\begin{proof}
Denote as $B_{d,k}$ a Bethe tree of degree $d$ and depth $k$, where the root has a degree of $d$, vertices at levels from 2 to $k-1$ have a degree of $d+1$, and vertices at the $k$'th level have a degree of 1. Let $n = n(d, k)$ be a number of vertices of $B_{d, k}$. Consider a Bethe tree $B_{d, k}$. By simple calculations we get $n=\frac{d^k-1}{d-1}$. Suppose $k\ge 2$, $d\ge 3$, thus using Theorem~2 and Theorem~3  from \cite{tight_bethe} and considering the asymptotic behavior, we obtain that $\exists d_0:\forall d\ge d_0$ 
\begin{equation*}
\frac{(d-1)^2}{d^k-1} \le \lambda_{n-1}(L(B_{d,k})) \le 2\frac{(d-1)^2}{d^k-1}.
\end{equation*}

\begin{figure}[ht]
    \centering
    \includegraphics[width=0.5\textwidth]{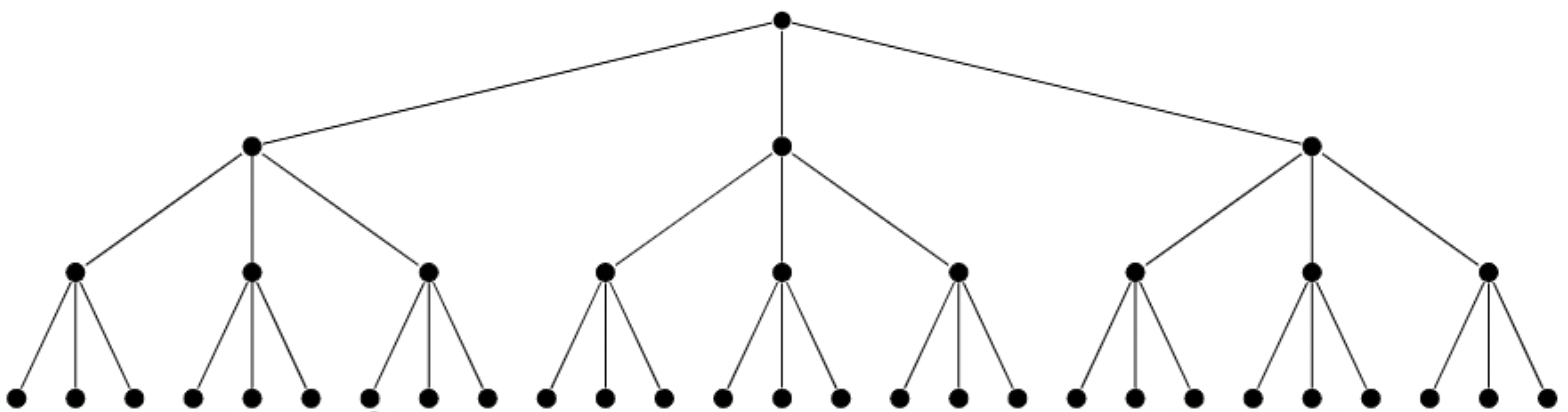}
    \caption{Example of $B_{3, 4}$}
    \label{fig:bethe_tree_example}
\end{figure}

Using results from \cite{stevanovic_spectr_gap} we get $d+2 \le \lambda_1(L(B_{d,k})) \le (\sqrt{d}+1)^2$, thus we conclude that $\exists d_1:\forall d \ge d_1, k \ge 2$

\begin{equation}
 \frac{n(B_{d,k})}{2}\le\chi(B_{d,k})\le2n(B_{d,k}).\label{chi_bounds}
\end{equation}

Denote as $\cV_1$ the set of vertices of type 1, $\cV_2$ the set of vertices of type 2 ($\cV_1\cap \cV_2 = \varnothing$), and $\cW$ the set of remaining vertices. Let $d\ge t > 2$. $\cV_1$ consists of $[\frac{d}{t}]$ subtrees with roots adjacent to the root of $B_{d, k}$, and $\cV_2$ is defined in the same way. Therefore $|\cV_1|=|\cV_2|=[\frac{d}{t}]\frac{d^{k-1}-1}{d-1}$.

\begin{figure}[ht]
    \centering
    \includegraphics[width=0.35\textwidth]{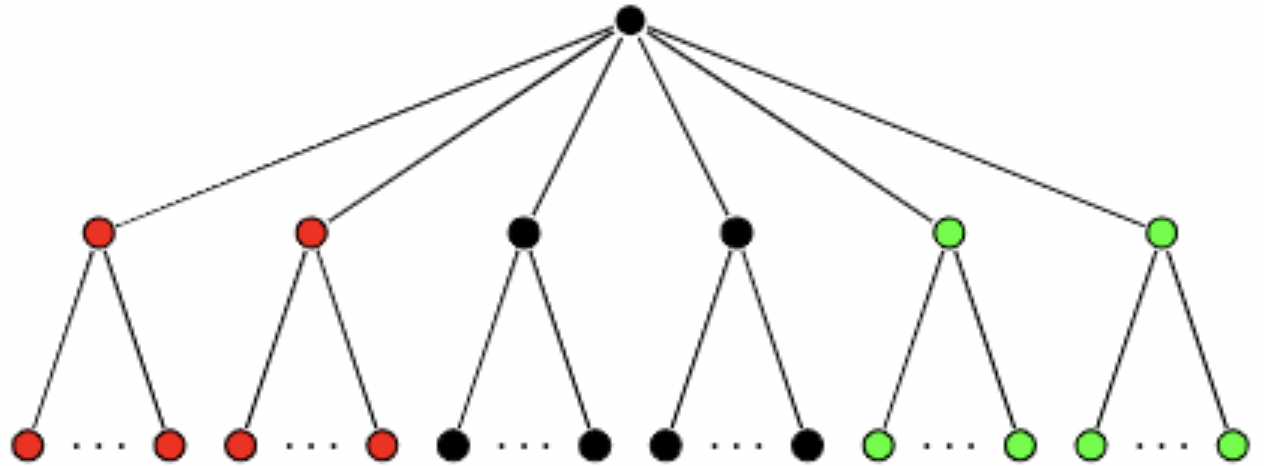}
    \caption{Example of splitting a graph into three sets. The vertices of $\cV_1$ ("Bad" vertices) are indicated in red. The vertices of $\cV_2$ ("Good" vertices) are indicated in green.}
    \label{fig:graph_partition_example}
\end{figure}

Denote the vertex functions $f_v: \ell_2\rightarrow \R$ depending on vertex type:
\begin{equation}\label{poly_func}
f_v(x)=\begin{cases}
 \frac{\mu}{2n}\norm{x}^2+\frac{L-\mu}{4|\cV_1|}\[(x_1-1)^2+\sum_{k=1}^{\infty}(x_{2k}-x_{2k+1})^2\], & v\in V_1 \\
 \frac{\mu}{2n}\norm{x}^2+\frac{L-\mu}{4|\cV_2|}\sum_{k=1}^{\infty}(x_{2k-1}-x_{2k})^2, &  v\in V_2\\
 \frac{\mu}{2n}\norm{x}^2, & v\in W
\end{cases}.
\end{equation}

Estimate $\cV_1$ and $\cV_2$ through $n$, using that $d \ge t$, $k\ge 2,d\ge 3$ we get
\begin{equation}
 |\cV_1|=|\cV_2|\ge \frac{d}{2t}\frac{d^{k-1}-1}{d-1}\ge \frac{n}{4t}. \label{werwe}
\end{equation}

Estimate the network's global characteristic number using the local one
\begin{equation*}
 \kappa_l = \frac{\frac{L-\mu}{2|\cV_1|}+\frac{\mu}{n}}{\frac{\mu}{n}} \le \frac{\frac{2t(L-\mu)}{n}+\frac{\mu}{n}}{\frac{\mu}{n}}=\frac{2t(L-\mu)+\mu}{\mu}=2(\kappa_g-1)t+1,
\end{equation*}
 thus we have
\begin{equation}\label{kappa_rels}
 \kappa_g \ge \frac{\kappa_l - 1}{2t}+1.
\end{equation}

Let us now start describing the sequence of edges $\{\cE_i\}_{i=1}^{\infty}$. Firstly, we introduce the $\swap(v_1, v_2)$ operation. When applied to the graph $\cG=(\cV, \cE)$, it changes the edges between $v_1$ and $v_2$. Note that this operation changes no more than $2(\deg(v_1) + \deg(v_2))$ edges ($\Delta \le 2(\deg(v_1) + \deg(v_2))$).

Next, introduce the following scheme: A graph consists of "good" and "bad" vertices. At each iteration, each good vertex adjacent to at least one bad vertex becomes a bad vertex. After that, we somehow change the edges in the graph, and the scheme continues. In our case, we call $\cV_1$ vertices "bad" and $\cG \setminus \cV_1$ vertices "good". Our goal is to make the $\cV_2$ vertices remain "good" as long as possible. The algorithm returns a graph sequence of decentralized problems $\{\DP_i\}_{i=1}^{\infty}$. The individual steps of the algorithm work as follows:

\begin{algorithm}[H]\caption{Inner loop of graph changing scheme with polynomial restrictions}\label{alg_chang_schm}
\begin{algorithmic}
	\STATE \textbf{Input:} {$\cG$, $\cV_1$, $\cV_2$, $i$}

	\STATE $B = \cV_1$

	\WHILE{$\NoBadVertices(\cV_2, B)$}
	\STATE $U = \PotentialBadVertices(B)$

	\FOR{$j=0,1,\dots,|U| - 1$}
		\STATE $u=U[j]$
		\STATE $v = \FindCandidate(u)$
		\STATE $B = B\cup \{v\}$
		\STATE $\cG = \swap(\cG, u, v)$
	\ENDFOR
	\STATE $\cG_i = \cG$
	\STATE $i=i+1$

	\ENDWHILE
\end{algorithmic}
\end{algorithm}

We make a few assumptions:
\begin{itemize}
\item Each level of the Bethe tree has a natural order of vertices, particularly each vertex has an ordered set of its children. The root children in $\cV_1$ are the most minimal, while the root children in $\cV_2$ are the most maximal.
\item $B$ is the set of "bad" vertices. At moment $i$, $\cG_i$ is a graph in the sequence $\{\DP\}_{i=1}^{\infty}$.
\item $i$ is the state number of the $\{\DP\}_{i=1}^{\infty}$ sequence on which we run Algorithm~\ref{alg_chang_schm}.
\end{itemize}
Each iteration works as follows:
\begin{itemize}
\item $\NoBadVertices$ checks if the set consists only of "good" vertices.
\item $\PotentialBadVertices$ computes the set of "good" vertices which would become "bad" ones after an iteration.
\item $\FindCandidate$ finds a vertex, which would be swapped with the input vertex. The step works as follows: we start with root vertex and move to the leftmost (lowest numbered) "good" vertex (not in $B$), if there is no such vertex, we stop. Return the vertex where we stopped.
\item $\swap$ operation swaps the edges of two input vertices, but not the vertices themselves.
\end{itemize}

Note that $\cV_1$ is the set of "bad" vertices, which can be obtained through several While iterations of Algorithm~\ref{alg_chang_schm}, starting from a graph with no "bad" vertices.

\begin{figure}[ht]
    \centering
    \includegraphics[width=0.5\textwidth]{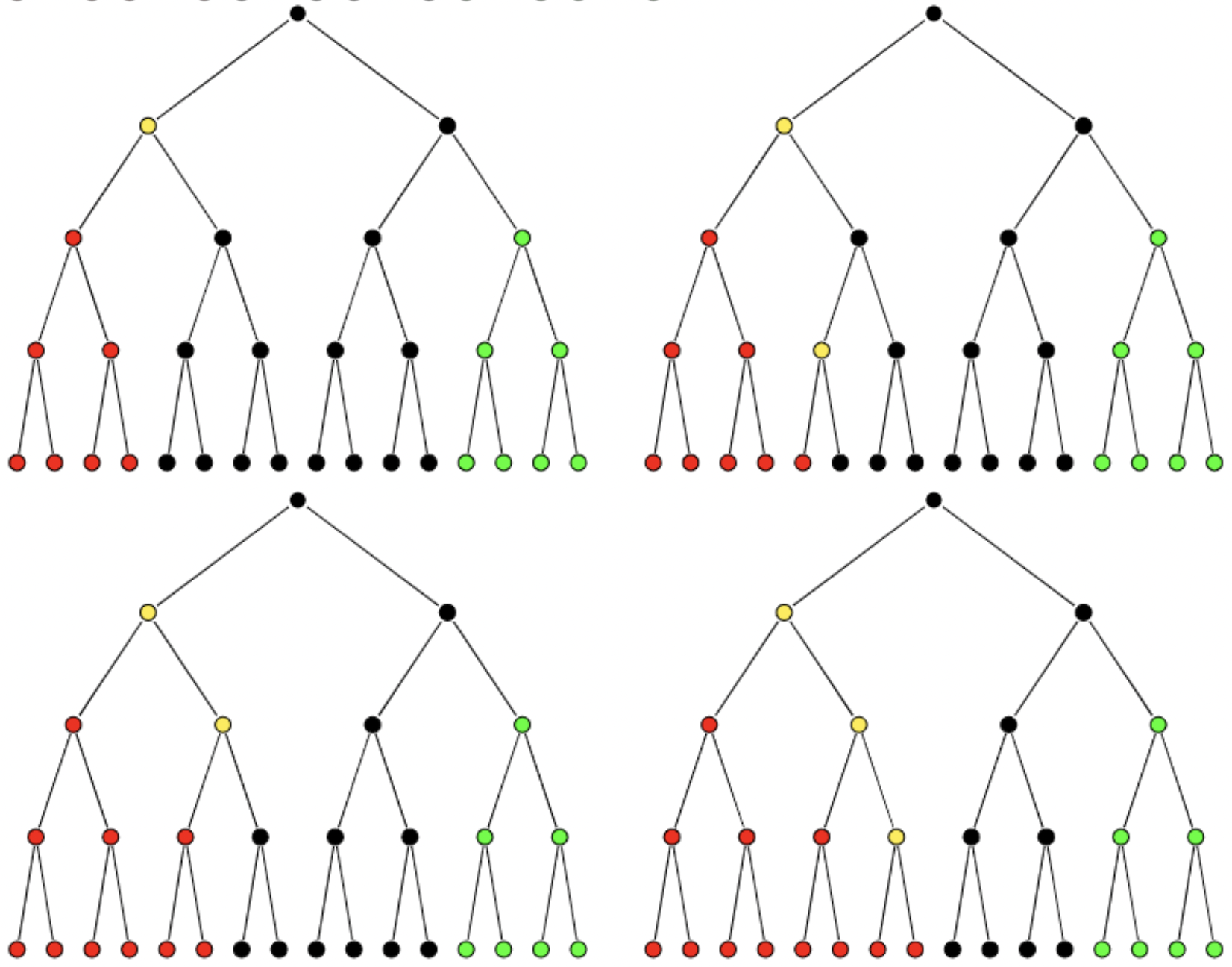}
    \caption{Example of a graph change scheme. Potentially "bad" vertices, i.e. those that will become "bad" in one move, are indicated in yellow.}
    \label{fig:changing_scheme_example}
\end{figure}

\begin{lemma}\label{change_depth}
 Consider an Algorithm~\ref{alg_chang_schm}, applied to $B_{d, k}$. At any moment of this algorithm, the number of new bad vertices is bounded as follows:
 \begin{equation}
     |U| \le k - 1.
 \end{equation}
\end{lemma}
\begin{proof}
 Note that such iterations support the following invariant: if vertex $p$ is "bad", then all its ancestors are "bad" too (those that can be reached from $p$ by going from the root). Consider the set of new "bad" vertices. On the last level, there are no such vertices, as it would be a contradiction with the first sentence. At other levels, there are at most two. Suppose the contradiction, let them be vertices $v$ and $u$. These vertices are adjacent to "bad" vertices $v'$ and $u'$ at the level below. Consider their common ancestor $p$. It is "good" and has two "good" children $v$ and $u$, which have "bad" vertices $v'$ and $u'$ in their subtrees. But then the algorithm would have to "fill" one of the subtrees on $v$ or $u$ (the one with the lowest number) with "bad" vertices, a contradiction.
\end{proof}
As a result, at each iteration, we get the upper bound on the change of edges:
\begin{equation}\label{row_change_upper}
 \Delta_i \le 4|U|d_{max}=4(k-1)(d+1).
\end{equation}
Using that $n=\frac{d^k-1}{d-1}$ and assuming $d\ge k$ we get
\begin{equation}
 \Delta_i \le 4(k-1)(d+1)\le 4kd \le 4kn^{1/(k-1)}.\label{change_upper}
\end{equation}

Let $T$ be the number of iterations needed for one of the vertices in $\cV_2$ to become "bad" (it equals the number of iterations of Algorithm~\ref{alg_chang_schm}), we will refer to this value as the "information flow". At each iteration, there are at most $k-1$ new "bad" vertices. Therefore, we get:
\begin{equation}\label{information_flow_T}
 T\ge \[\frac{|\cW|}{k - 1}\]+1\ge\frac{|\cW|}{k - 1}.
\end{equation}

Estimate the size of the neutral vertex set $\cW$ using $n=\frac{d^k-1}{d-1}$ and the definitions of $\cV_1, \cV_2$
\begin{equation*}
 |\cW|=n-|\cV_1|-|\cV_2| \ge \frac{1}{d-1}\(d^k-1-2\frac{d^k-d}{t}\) \ge \(1-\frac{2}{t}\)n.
\end{equation*}

Therefore using \ref{chi_bounds} we get lower bound on the information flow:
\begin{equation}\label{flow_lower}
 T\ge \frac{1-2/t}{2(k-1)}\chi.
\end{equation}

Now, we are going to define the whole sequence of graphs in $\DP_j$ for every $j\in\N$. Algorithm~\ref{alg_chang_schm} only defines the graphs for $i=i_0,i_0+1,\dots,i_0+T-1$, where $i_0$ is the input state number for the Inner algorithm. We run Algorithm~\ref{alg_chang_schm} iteratively. After each iteration, at least one vertex of $\cV_2$ becomes "bad", then we rearrange the sets $\cV_1$ and $\cV_2$, reverse the order of children for each vertex in $\cG$, and run the algorithm again.

\begin{algorithm}[H]\caption{Outer loop of graph changing scheme with polynomial restrictions}\label{alg_outer_chang_schm}
\begin{algorithmic}
	\STATE \textbf{Input:} {$\cG$, $\cV_1$, $\cV_2$}
	\STATE $i=1$
	\WHILE{True}
		\STATE $\InnerLoop(\cG, \cV_1, \cV_2, i)$
		\STATE $\Rearrange(\cV_1, \cV_2)$
		\STATE $\ReverseOrders(\cG)$
	\ENDWHILE
\end{algorithmic}
\end{algorithm}

$\InnerLoop$ is Algorithm~\ref{alg_chang_schm}. $\Rearrange$ changes the pointers for variables $\cV_1$ and $\cV_2$. $\ReverseOrders$ reverses the order of children for each vertex.

We have found a sequence of graphs in which information flows slowly. Specifically, to get from $\cV_1$ to $\cV_2$, it takes $T$ iterations. To get back (from $\cV_2$ to $\cV_1$) it takes $T$ iterations as well and so on.

Let $x_0=0$ be the initial point for the first-order decentralized algorithm. For every $m\ge 1$, we define $l_m=\min{p\ge 1|\exists v: \exists x\in \cH_{v}(p): x_m \neq 0}$ as the first moment when we can get a non-zero element at the $m$-th place at any node.

Considering the types of functions on vertices of the graph \ref{poly_func}, we can conclude that functions on vertices from $\cV_1$ can "transfer" (by calculating the gradient) information (non-zero element) from the even positions ($2, 4, 6,\ldots$) to the next ones, and functions on vertices from $\cV_2$ can transfer information from the odd positions ($1, 3, 5,\ldots$) to the next ones. Therefore, for the network to get a new non-zero element at the next position, a complete iteration of Algorithm~\ref{alg_chang_schm} is required, that is $T$ communication iterations.

One of the main ideas is that this "information" cannot spread faster than "bad" vertices. 

To reach the $m$-th non-zero element, we need to make at least $m$ local steps and $(m - 1)T$ communication steps to transfer information from gradients between $\cV_1$ and $\cV_2$ sets. Therefore, we can estimate $l_m$:
\begin{equation}\label{non_zero_inf}
 l_m \ge (m-1)T+m.
\end{equation}

The solution of the global optimization problem is $x^*_p = \(\frac{\sqrt{\kappa_g} - 1}{\sqrt{\kappa_g} +1}\)^p$.

For any $m, p$ such that $l_m > p$
\begin{equation*}
 \norm{x_p-x_*}^2 \ge (x_*)_m^2+(x_*)_{m+1}^2+\ldots=\(\frac{\sqrt{\kappa_g} - 1}{\sqrt{\kappa_g} +1}\)^m\norm{x_0-x_*}^2.
\end{equation*}

Using \ref{non_zero_inf} we can take $m=\ceil{\frac{p}{T+1}}+1$. From \ref{kappa_rels} we conclude that $\frac{\sqrt{\kappa_g} - 1}{\sqrt{\kappa_g} +1} \ge 1 - \frac{2\sqrt{6}}{\sqrt{\kappa_l}}$

Therefore using \ref{kappa_rels}, \ref{flow_lower} and assign $t=3$ we get
\begin{equation*}
 \norm{x_p-x_*}^2 \ge \(\frac{\sqrt{\kappa_g} - 1}{\sqrt{\kappa_g} +1}\)^{\ceil{p/\(\frac{1}{6(k-1)}\chi+1\)}+1}\norm{x_0-x_*}^2.
\end{equation*}

Rearranging it, we get
\begin{equation}\label{arg_lower}
 \norm{x_p-x_*}^2 \ge \(\max\left\{1-2\sqrt{6}\sqrt{\frac{\mu}{L}}\right\}\)^{\frac{6(k-1)p}{\chi}+2}\norm{x_0-x_*}^2.
\end{equation}

\end{proof}

\section{Proof of the Theorem~\ref{theorem:log}}

\begin{proof}
The proof is very similar to the proof of the Theorem~\ref{theorem:poly}, but here we fix $d=2$ and $k\to \infty$.

Let $B_k$ be a binary tree $B_{2, k}$. Using the lower and upper bounds on algebraic connectivity ($\lambda_{n-1}$) of such trees from \cite{binary_tree_alg_connect} and following the same logic as in \ref{chi_bounds}, we can conclude that there exists a $k_0$ such that for all $k>k_0$ the following is true
\begin{equation}
 2n(B_k)\le\chi(B_k)\le 6n(B_k).\label{chi_bounds_bin_tree}
\end{equation}
Denote the set of vertices of type 1 as $\cV_1$, the set of vertices of type 2 as $\cV_2$ ($\cV_1\cap \cV_2 =\varnothing$), and the set of remaining vertices as $\cW$. Suppose that every non-leaf vertex has a "left" and "right" child. Let $\cV_1$ be a subtree with a "left-left" root vertex (it can be reached from the graph's root by going to the left child and then back to the left child). $\cV_2$ is defined in the same way. Therefore $|\cV_1|=|\cV_2|=2^{k-2}-1$.

Denote the vertex functions $f_v: \ell_2\rightarrow \R$ similarly as in \ref{poly_func}.

Estimate $\cV_1$ and $\cV_2$ through $n$, using that $k\ge 4$ we get
\begin{equation}
 \frac{n}{5}\le |\cV_1|=|\cV_2|\le\frac{n}{4}.
\end{equation}

Estimate global characteristic number of the network through local one
\begin{equation*}
 \kappa_l = \frac{\frac{L-\mu}{2|\cV_1|}+\frac{\mu}{n}}{\frac{\mu}{n}} \le \frac{\frac{5(L-\mu)}{2n}+\frac{\mu}{n}}{\frac{\mu}{n}}=\frac{5(L-\mu)+\mu}{2\mu}=\frac{5}{2}(\kappa_g-1)+1,
\end{equation*}
 thus we have
\begin{equation}\label{kappa_rels_log}
 \kappa_g \ge \frac{2}{5}(\kappa_l-1)+1.
\end{equation}

Next, we will use exactly the same technique to construct a sequence of communication graphs as in the proof of the Theorem~\ref{theorem:poly} (Algorithm~\ref{alg_chang_schm} and Algorithm~\ref{alg_outer_chang_schm}). As a result, we get something resembling \ref{information_flow_T} inequality on "information flow" defined in previous proof
\begin{equation}
 T\ge\frac{|\cW|}{k - 1}.
\end{equation}
Using $n=2^k-1$ and the definitions of $\cV_1, \cV_2$, we can estimate the size of the neutral vertex set $\cW$. 
\begin{equation*}
 |\cW|=n-|\cV_1|-|\cV_2| = 2^k-1-2(2^{k-2}-1) \ge \frac{n}{2}.
\end{equation*}
By using \ref{chi_bounds_bin_tree} and inequality $k-1\le \log_2 n$ we derive a lower bound on $T$:
\begin{equation}\label{flow_lower_log}
 T\ge\frac{\chi}{12(k-1)}\ge\frac{\chi}{12log_2(\chi/2)}.
\end{equation}
Also we similarly take an upper bound on edge change $\Delta_i$ in graph sequence
\begin{equation}\label{change_upper_log}
 \Delta_i \le 4|U|d_{max}=12(k-1)\le 12\log_2 n .
\end{equation}

Do the same reasoning with $l_m$ (defined in the last section).

In order to reach the $m$-th non-zero element, at least $m$ local steps and $(m - 1)T$ communication steps are required to transfer information from gradients between $\cV_1$ and $\cV_2$ sets. Based on this, we can estimate $l_m$:
\begin{equation}\label{non_zero_inf_log}
 l_m \ge (m-1)T+m.
\end{equation}

The solution of the global optimization problem is $x^*_p = \(\frac{\sqrt{\kappa_g} - 1}{\sqrt{\kappa_g} +1}\)^p$.

For any $m, p$ such that $l_m > p$
\begin{equation*}
 \norm{x_p-x_*}^2 \ge (x_*)_m^2+(x_*)_{m+1}^2+\ldots=\(\frac{\sqrt{\kappa_g} - 1}{\sqrt{\kappa_g} +1}\)^m\norm{x_0-x_*}^2.
\end{equation*}

Using \ref{non_zero_inf_log} we can take $m=\ceil{\frac{p}{T+1}}+1$. From \ref{kappa_rels_log} we conclude that $\frac{\sqrt{\kappa_g} - 1}{\sqrt{\kappa_g} +1} \ge 1 - \frac{\sqrt{10}}{\sqrt{\kappa_l}}$

Therefore using \ref{kappa_rels_log}, \ref{flow_lower_log} we get
\begin{equation*}
 \norm{x_p-x_*}^2 \ge \(\max\left\{0,1 - \sqrt{10}\sqrt{\frac{\mu}{L}}\right\}\)^{\ceil{p/\(\frac{\chi}{12\log_2(\chi/2)}+1\)}+1}\norm{x_0-x_*}^2.
\end{equation*}

Rearranging it, we get
\begin{equation}\label{arg_lower_log}
 \norm{x_p-x_*}^2 \ge \(\max\left\{0,1 - \sqrt{10}\sqrt{\frac{\mu}{L}}\right\}\)^{\frac{12\log_2(\chi/2)p}{\chi}+2}\norm{x_0-x_*}^2.
\end{equation}
\end{proof}

\section{Proof of the Theorem~\ref{theorem:const}}
Firstly, let's define the structure of the graph that will serve as a counter-example in the case under consideration and study its properties.

We will define the graph $H_{d, k}$ through induction. Let $H_{1, k}$ be a path of length $k$, and call any of its leaf vertices the root. Then, assuming we have defined $H_{d, k}$ for all $k$, we define $H_{d+1, k}$ as follows: take a path of length $k$ (and call the leaf vertex of the path the root of the graph), and attach a copy of $H_{d, k}$ to the root of each vertex in the path.

It can be seen that each tree $H_{d, k}$ consists of paths of length $k$ with fixed vertices, which we will refer to as roots. Consider one such path, with a start vertex and an end vertex, where the end vertex is a root and the start vertex is the other leaf vertex. We will also assign a number from $1$ to $k$ to each vertex in the path, corresponding to the distance from the start vertex, increased by $1$. That is, the start and root vertices have numbers $1$ and $k$ respectively.

We will divide the tree $H_{d, k}$ into levels. As the graph was constructed through induction, the first level will consist of the vertices added in the first iteration of the induction, the second level will consist of the vertices added in the second iteration of induction, and so on, up to level $d$.

To move forward, let's assign coordinates to these vertices as follows: consider the vertex $v$ at level $g$, then its coordinates will be the tuple $x = (x_1, \ldots, x_g)$, where $x_i$ is the number corresponding to the vertex closest to $v$ at level $i$.

To proceed, let's introduce a linear order relation on these vertices as follows: if they have different lengths and the coordinates of the first vertex is a prefix of the second one, then the second one is considered smaller. Otherwise, the one with the smallest element that has the first difference from left to right is considered smaller. This way, the vertex $(x_1,\ldots,x_g)$ will be adjacent to the vertices $(x_1,\ldots,x_g - 1)$, $(x_1,\ldots,x_g + 1)$, $(x_1,\ldots,x_g, d)$ if notations is correct and to vertex $(x_1,\ldots,x_{g - 1})$ if $x_g=d$."

Let us introduce a linear order relation on these vertices: if the coordinate of the first vertex is the prefix of the second one, then the second one is smaller, otherwise the one with the smallest element, which has the first difference from left to right, is smaller. Then the vertex $(x_1,\ldots,x_g)$ will be adjacent to the vertices $(x_1,\ldots,x_g - 1)$, $(x_1,\ldots,x_g + 1)$, $(x_1,\ldots,x_g, d)$ if notations is correct and to vertex $(x_1,\ldots,x_{g - 1})$ if $x_g=d$.

Consider a graph $H_{d, k}$.

For this graph $n=k+k^2+\ldots+k^d$. Let $D$ be the diameter of this graph, it can be easily seen that $D=(2d-1)k$. According to Theorem~4.1.1 in \cite{laplacian_thesis}, we can obtain an estimation of $\lambda_{n-1}(L(H_{d, k}))$
\begin{equation}\label{lambda_min_lower:const}
    \lambda_{n-1}(L(H_{d, k})) \ge \frac{1}{nD}\ge\frac{1}{d(2d-1)k^{d+1}}.
\end{equation}

Using results from \cite{stevanovic_spectr_gap} we get
\begin{equation}\label{lambda_max_upper:const}
    4 \le \lambda_1(L(H_{d,k})) \le 3+2\sqrt{2}\le 6.
\end{equation}

As a result, by using \ref{lambda_min_lower:const} and \ref{lambda_max_upper:const} we can obtain an upper bound for $\chi(H_{d, k})$.
\begin{equation}\label{chi_upper:const}
    \chi(H_{d, k}) \le 6d(2d-1)k^{d+1} \le 6d(2d-1)n^{\frac{d+1}{d}}.
\end{equation}

Let $\cV_1$ and $\cV_2$ be disjoint sets of vertices of type 1 and type 2, respectively, and let $\cW$ be the set of remaining vertices. $\cV_1$ consists of vertices that have a coordinate of at most $(\[\frac{k}{3}\])$, and $\cV_2$ consists of vertices that have a coordinate of at least $(k-\[\frac{k}{3}\])$. Therefore, $|\cV_1|=|\cV_2|=\[\frac{k}{3}\](k+k^2+\ldots+k^{d-1})$.

Denote the vertex functions $f_v: \ell_2\rightarrow \R$ depending on vertex type:
\begin{equation}\label{poly_func}
f_v(x)=\begin{cases}
 \frac{\mu}{2n}\norm{x}^2+\frac{L-\mu}{4|\cV_1|}\[(x_1-1)^2+\sum_{k=1}^{\infty}(x_{2k}-x_{2k+1})^2\], & v\in V_1 \\
 \frac{\mu}{2n}\norm{x}^2+\frac{L-\mu}{4|\cV_2|}\sum_{k=1}^{\infty}(x_{2k-1}-x_{2k})^2, &  v\in V_2\\
 \frac{\mu}{2n}\norm{x}^2, & v\in W
\end{cases}.
\end{equation}

Estimate $\cV_1$ and $\cV_2$ through $n$, let $k\ge 3$
\begin{equation}
 |\cV_1|=|\cV_2|\ge \frac{k}{6}(k+k^2+\ldots+k^{d-1})=\frac{n}{12}. \label{werwe:const}
\end{equation}

Estimate the network's global characteristic number using the local one
\begin{equation*}
 \kappa_l = \frac{\frac{L-\mu}{2|\cV_1|}+\frac{\mu}{n}}{\frac{\mu}{n}} \le \frac{\frac{6(L-\mu)}{n}+\frac{\mu}{n}}{\frac{\mu}{n}}=\frac{6(L-\mu)+\mu}{\mu}=6(\kappa_g-1)+1,
\end{equation*}
 thus we have
\begin{equation}\label{kappa_rels_const}
 \kappa_g \ge \frac{\kappa_l - 1}{6}+1.
\end{equation}

Let us now describe the sequence of edges $\{\cE_i\}_{i=1}^{\infty}$. Similar to the proof of Theorem~\ref{theorem:poly}, we will construct an algorithm that generates a sequence of graphs that works under the same conditions as Algorithm~\ref{alg_chang_schm} and Algorithm~\ref{alg_outer_chang_schm}. In this scheme, we will refer to vertices in $\cV_1$ as "bad" and the remaining vertices as "good". After each iteration, a "good" vertex that is adjacent to a "bad" vertex becomes "bad", and the graph is modified in some way. The goal is to keep the vertices in $\cV_2$ "good" for as long as possible. Additionally, we will maintain the invariant that after each graph change, a "good" vertex cannot be less than a "bad" vertex.

\begin{algorithm}[H]\caption{Inner loop of graph changing scheme with constant restrictions}\label{alg_chang_schm:const}
\begin{algorithmic}
	\STATE \textbf{Input:} {$\cG$, $\cV_1$, $\cV_2$, $i$}

	\STATE $B = \cV_1$

	\WHILE{$\NoBadVertices(\cV_2, B)$}
	\STATE $U = \PotentialBadVertices(B)$
        \STATE $\AtLastLevel(U, B)$

	\FOR{$j=0,1,\dots,|U| - 1$}
		\STATE $u=U[j]$
		\STATE $v = \FindCandidate(u, B)$
		\STATE $B = B\cup \{v\}$
		\STATE $\cG = \swap(\cG, u, v)$
	\ENDFOR
	\STATE $\cG_i = \cG$
	\STATE $i=i+1$

	\ENDWHILE
\end{algorithmic}
\end{algorithm}

Every function works in the same way as in the Algorithm~\ref{alg_chang_schm}, except for $\FindCandidate$ and $\AtLastLevel$. $\FindCandidate$ finds a vertex that would be swapped with the input vertex. It finds the smallest "good" vertex in the graph. It is simple to check that the invariant is preserved. $\AtLastLevel$ checks if there is a vertex at the last level, makes it "bad", and removes it from $U$.

Then we apply Algorithm~\ref{alg_outer_chang_schm}, but using Algorithm~\ref{alg_chang_schm:const} as the inner algorithm, thus obtaining a sequence of graphs.

\begin{lemma}\label{change_depth:const}
 Consider an Algorithm~\ref{alg_chang_schm:const}, applied to $H_{d, k}$. At any moment of this algorithm, the number of new bad vertices is bounded as follows:
 \begin{equation}
     |U| \le d.
 \end{equation}
\end{lemma}
\begin{proof}
We consider the set $U$ and show that it can contain at most one element from each level. Suppose the converse, let the vertices $v,u \in U$ and belong to level $g$. Let vertex $v$ have coordinates $(x_1,\ldots,x_g)$ and vertex $u$ have coordinates $(y_1,\ldots,y_g)$ and $v<u$. The vertex $u$ is adjacent to the "bad" vertex $b$, and $b<u$, so it can be $(y_1,\ldots,y_g-1)$ or $(y_1,\ldots,y_g,d)$. The second case is impossible because the "bad" vertex $(y_1,\ldots,y_g,d)$ is greater than $v$, and this contradicts the invariant. Consider the first case when $b$ has coordinates $(y_1,\ldots,y_g-1)$. $v<u$, so $v \leq b$, but they cannot be equal, since at the given iteration, $v$ is only potentially "bad" (i.e. "good") so far, so we are led to the same contradiction when the "bad" vertex is greater than the "good" one.
\end{proof}

Note that either $U$ contains no vertices on the last level, in which case $|U|\le d-1$ (as can be proved similarly to Lemma~\ref{change_depth:const}), or there is a vertex on the last level, but it need not be swapped. Therefore, we can obtain an upper bound on the number of edges changed at each iteration.
\begin{equation}\label{row_change_upper:const}
 \Delta_i \le 12(d - 1).
\end{equation}

Let $T$ be the number of iterations, needed for one of the vertex in $V_2$ to become bad (it equals to the number of iterations of Algorithm~\ref{alg_chang_schm:const}), we wil refer to this value as "information flow". According to Lemma~\ref{change_depth:const} at each While iteration there are not more than $d
$ new bad vertices, therefore we get
\begin{equation}\label{information_flow_T:const}
 T\ge \[\frac{|\cW|}{d}\]+1\ge\frac{|\cW|}{d}.
\end{equation}

Using $|\cV_1|=|\cV_2|\le \frac{n}{3}$ to estimate the size of the neutral vertex set $cW$, we get 
\begin{equation*}
 |\cW|=n-|\cV_1|-|\cV_2| \ge \frac{n}{3}.
\end{equation*}

Therefore using \ref{chi_upper:const} we get lower bound on the information flow:
\begin{equation}\label{flow_lower:const}
 T\ge \frac{n}{3d} \ge \frac{\chi^{\frac{d}{d+1}}}{3d(6d(2d-1))^{\frac{d}{d+1}}}.
\end{equation}

To proceed further, we apply a similar approach to determine $l_m$ (defined in the the proof of the Theorem~\ref{theorem:poly}).

To reach the $m$-th non-zero element, we need to make at least $m$ local steps and $(m - 1)T$ communication steps to transfer information from gradients between $\cV_1$ and $\cV_2$ sets. Therefore, we can estimate $l_m$:
\begin{equation}\label{non_zero_inf_const}
 l_m \ge (m-1)T+m.
\end{equation}

The solution of the global optimization problem is $x^*_p = \(\frac{\sqrt{\kappa_g} - 1}{\sqrt{\kappa_g} +1}\)^p$.

For any $m, p$ such that $l_m > p$
\begin{equation*}
 \norm{x_p-x_*}^2 \ge (x_*)_m^2+(x_*)_{m+1}^2+\ldots=\(\frac{\sqrt{\kappa_g} - 1}{\sqrt{\kappa_g} +1}\)^m\norm{x_0-x_*}^2.
\end{equation*}

Using \ref{non_zero_inf_const} we can take $m=\ceil{\frac{p}{T+1}}+1$. From \ref{kappa_rels_const} we conclude that $\frac{\sqrt{\kappa_g} - 1}{\sqrt{\kappa_g} +1} \ge 1 - \frac{2\sqrt{6}}{\sqrt{\kappa_l}}$

Let $C(d)=3d(6d(2d-1))^{\frac{d}{d+1}}$. Therefore using \ref{kappa_rels_const}, \ref{flow_lower:const} we get
\begin{equation*}
 \norm{x_p-x_*}^2 \ge \(\max\left\{0, 1 - 2\sqrt{6}\sqrt{\frac{\mu}{L}}\right\}\)^{\ceil{p/\(C(d)^{-1}\chi^{\frac{d}{d+1}}+1\)}+1}\norm{x_0-x_*}^2.
\end{equation*}

Rearranging it, we get
\begin{equation}\label{arg_lower_log}
 \norm{x_p-x_*}^2 \ge \(\max\left\{0, 1 - 2\sqrt{6}\sqrt{\frac{\mu}{L}}\right\}\)^{C(d)p\chi^{-\frac{d}{d+1}}+2}\norm{x_0-x_*}^2.
\end{equation}

\section{Accelerated Method over Time-Varying Function}\label{sec:acc_method_over_tw_function}

In this section, we show the convergence of accelerated Nesterov method over a uniformly non-increasing time-varying function. The proof is based on potential analysis in \cite{bansal2019potential}, and a similar proof technique was used in \cite{rogozin2019optimal}.

Consider a sequence of functions $\{f_k(x)\}_{k=0}^\infty$ such that the following assumptions hold.
\begin{assumption}\label{assum:strongly_convex_smooth_function_sequence}
	For every $k = 0, 1, 2, \ldots$, function $f_k(x)$ is $L$-smooth and $\mu$-strongly convex, that is, for any $x, y\in\R^d$ we have
	\begin{align*}
	\frac{\mu}{2}\norm{y - x}_2^2\leq f(y) - f(x) - \angles{\nabla f(x), y - x}\leq \frac{L}{2}\norm{y - x}_2^2.
	\end{align*}
\end{assumption}
\begin{assumption}\label{assum:nonincreasing_function_sequence}
	Sequence $\{f_k(x)\}_{k=0}^\infty$ is uniformly non-increasing: for each $x\in\R^d$ and for any $k = 0, 1, 2, \ldots$ we have
	\begin{align*}
	f_{k+1}(x)\leq f_k(x).
	\end{align*}
\end{assumption}
\begin{assumption}\label{assum:function_sequence_common_minimizer}
	Functions of sequence $\{f_k(x)\}_{k=0}^\infty$ have a common minimizer $x^*$.
\end{assumption}

Let an accelerated method be run over $\{f_k(x)\}_{k=0}^\infty$.
\begin{subequations}\label{eq:nesterov_method_time_varying_function_sequence}
	\begin{align}
	y_{k+1} &= x_k - \frac{1}{L}\nabla f_k(x_k), \\
	x_{k+1} &= \left( 1 + \frac{\sqrt{\kappa} - 1}{\sqrt{\kappa} + 1}\right)  y_{k+1} - \frac{\sqrt{\kappa} - 1}{\sqrt{\kappa} + 1} y_k,
	\end{align}
\end{subequations}
where $\kappa = L / \mu$. Then we have the following convergence result.

\begin{theorem}\label{th:nesterov_method_time_varying_function_sequence}
	Let accelerated Nesterov method be run over sequence $\{f_k\}_{k=0}^\infty$ and let Assumptions~\ref{assum:strongly_convex_smooth_function_sequence} and \ref{assum:nonincreasing_function_sequence} hold. We have
	\begin{align*}
	f_N(y_N) - f^* &\leq \frac{(L + \mu)R^2}{2} (1 - 1/\sqrt{\kappa})^N, \\
	\norm{y_N - y^*}_2^2 &\leq \frac{(L+\mu)R^2}{\mu} (1 - 1/\sqrt{\kappa})^N.
	\end{align*}
\end{theorem}
Before passing to proof of Theorem~\ref{th:nesterov_method_time_varying_function_sequence}, we need the following auxiliary lemma.
\begin{lemma}\label{lem:auxiliary_x_z_nabla}
	Consider updates in \eqref{eq:nesterov_method_time_varying_function_sequence} and define 
	\begin{align*}
	\tau &= \frac{1}{\sqrt{\kappa}+1}, \text{ and }
	z_{k+1} = \frac{1}{\tau} x_{k+1} - \frac{1-\tau}{\tau} y_{k+1}.
	\end{align*}
	Then, $z_{k+1} = \frac{1}{1+\gamma}z_k + \frac{\gamma}{1+\gamma}x_k - \frac{\gamma}{\mu(1+\gamma)}\nabla f_k(x_k)$,
	where $\gamma = \frac{1}{\sqrt\kappa - 1}$.
\end{lemma}
\begin{proof}
	By the update rule for $x_{k+1}$ given in \eqref{eq:nesterov_method_time_varying_function_sequence} and the definition of $\tau$, we have that
	\begin{align*}
	x_{k+1} 
	&= \left( 1 + \frac{\sqrt\kappa - 1}{\sqrt\kappa + 1}\right)  y_{k+1} - \frac{\sqrt\kappa - 1}{\sqrt\kappa + 1} y_k \\ 
	&= (2 - 2\tau)y_{k+1} - (1 - 2\tau)y_k.
	\end{align*}
	
	Moreover, by the definition of $z_{k+1}$, it follows that
	\begin{align*}
	z_{k+1}  
	& =\frac{1}{\tau}x_{k+1} - \frac{1-\tau}{\tau}y_{k+1} \\
	& =\frac{1}{\tau}\left( (2-2\tau)y_{k+1} - (1-2\tau)y_k\right)  - \frac{1-\tau}{\tau}y_{k+1} \\
	& =\frac{1}{\tau} \left( (1-\tau)y_{k+1} - (1-2\tau)y_k\right) .
	\end{align*}
	
	Now we use the update rule for $y_{k+1}$ given in \eqref{eq:nesterov_method_time_varying_function_sequence} and also note that $x_k = (1-\tau)y_k + \tau z_k$:
	
	\begin{align*}
	z_{k+1}
	& =\frac{1}{\tau} \Big[ (1-\tau)(x_k - \frac{1}{L}\nabla f_k(x_k)) - \frac{1-2\tau}{1-\tau}(x_k - \tau z_k) \Big] \\
	& =\frac{1-2\tau}{1-\tau}z_k + \frac{\tau}{1-\tau}x_k - \frac{1-\tau}{L\tau}\nabla f_k(x_k)  \\
	& \overset{\circledOne}{=}\frac{\sqrt\kappa - 1}{\sqrt\kappa}z_k + \frac{1}{\sqrt\kappa}x_k - \frac{1}{\mu\sqrt\kappa} \\
	& \overset{\circledTwo}{=}\frac{1}{1+\gamma}z_k + \frac{\gamma}{1+\gamma}x_k - \frac{\gamma}{\mu(1+\gamma)}\nabla f_k(x_k),
	\end{align*}
	where $\circledOne$ is obtained by using the definitions of $\tau$ and $\kappa$, and $\circledTwo$ is obtained by using the definition of $\gamma$.
\end{proof}

\begin{lemma}\label{lem:potential_change}
	Let $\{f_k(x)\}_{k=0}^\infty$ be a sequence of functions for which Assumptions~\ref{assum:strongly_convex_smooth_function_sequence} and \ref{assum:nonincreasing_function_sequence} hold. Introduce potential function
	\begin{align}\label{eq:nesterov_potential}
	\Psi_k = (1 + \gamma)^k\cdot \left( f_k(y_k) - f^* + \frac{\mu}{2}\|z_k - x^*\|_2^2 \right),
	\end{align}
	Then, it holds that
	\begin{equation}\label{bound:delta_phi}
	\Delta\Psi_k = \Psi_{k+1} - \Psi_k\le 0.
	\end{equation}
\end{lemma}

\begin{proof}
	The proof is analogous to the proof in Section $5.4$ in \cite{bansal2017potential}. We use the definitions of $\tau, z_k$ given in Lemma \ref{lem:auxiliary_x_z_nabla}. We have
	
	\begin{align*}
	&\Delta\Psi_k\cdot(1+\gamma)^{-k} = (1+\gamma) \big( f(y_{k+1}) - f^* + \frac{\mu}{2}\|z_{k+1} - x^*\|_2^2 \big) \\
	&\qquad\qquad \qquad\qquad- \big( f(y_k) - f^* + \frac{\mu}{2} \|z_k - x^*\|_2^2 \big) \nonumber\\
	&\quad= (1+\gamma)\big( f_{k+1}(y_{k+1}) - f_{k+1}(x^*) \big) - \big(f_k(y_k) - f_k(x^*)\big)\\
	& \quad\quad+ \frac{\mu}{2} \Big[ (1+\gamma)\|z_{k+1} - x^*\|_2^2 - \|z_k - x^*\|_2^2 \Big]. \numberthis \label{eq:potential_change}
	\end{align*}
	
	Note that from Assumption~\ref{assum:nonincreasing_function_sequence} and from basic gradient step inequality we have
	\begin{align*}
	f_k(y_k)\leq f_k(y_{k+1}) & \le f_k(x_k) - \frac{1}{2L}\|\nabla f_k(x_k)\|_2^2,
	\end{align*}
	
	We bound th first term in \eqref{eq:potential_change} as follows:
	\begin{align*}
	& (1+\gamma)\big( f_{k+1}(y_{k+1}) - f_{k+1}(x^*)\big) - \big(f_k(y_k) - f_k(x^*) \big) \\
	&\quad\le (1+\gamma)\big( f_k(x_k) - \frac{1}{2L}\|\nabla f_k(x_k)\|_2^2 - f^* \big) - \big( f_k(y_k) - f^* \big) \\
	&\quad= f_k(x_k) - f_k(y_k) + \gamma(f_k(x_k) - f^*) - (1+\gamma)\frac{\|\nabla f_k(x_k)\|_2^2}{2L} \\  
	&\quad\le \langle\nabla f_k(x_k), x_k - y_k\rangle + \gamma\big(\langle\nabla f_k(x_k), x_k - x^*\rangle - \frac{\mu}{2}\|x_k - x^*\|_2^2\big)\\
	&\qquad - \frac{1+\gamma}{2L}\|\nabla f_k(x_k)\|_2^2 \numberthis\label{eq:potential_change_first_kerm}.
	\end{align*}
	
	Let us employ Lemma~\ref{lem:auxiliary_x_z_nabla} to get rid of references to $y_k$. We have
	\begin{align*}
	z_k &= 
	\big(\frac{1}{\tau} - 1\big)(x_k - y_k) + x_k = 
	\sqrt\kappa(x_k - y_k) + x_k \\
	\gamma(z_k - x^*) &= \sqrt\kappa\gamma(x_k - y_k) + \gamma(x_k - x^*).
	\end{align*}
	
	Note that $\sqrt\kappa\gamma = 1 + \gamma$. We have
	\begin{align*}
	(x_k - y_k) + \gamma(x_k - x^*) = \frac{1}{1 + \gamma}\cdot \Big[\gamma(z_k - x^*) + \gamma^2 (x_k - x^*)\Big].
	\end{align*}
	
	After that, we rewrite the expression on the right hand side of $\eqref{eq:potential_change_first_kerm}$ as follows:
	\begin{align*}
	\label{eq:potential_change_first_kerm_2}
	&\frac{1}{1+\gamma}\langle\nabla f_k(x_k), \gamma (z_k - x^*) + \gamma^2(x_k - x^*)\rangle - \\
	&\quad \frac{\mu\gamma}{2}\|x_k - x^*\|_2^2 
	- \frac{1 + \gamma}{2L}\|\nabla f_k(x_k)\|_2^2. \numberthis
	\end{align*}
	
	
	We bound the second term in \eqref{eq:potential_change} similarly to \cite{bansal2019potential}. By Lemma \ref{lem:auxiliary_x_z_nabla}:
	
	\begin{align*}
	&\frac{\mu}{2} \Big[ (1+\gamma)\|z_{k+1} - x^*\|_2^2 - \|z_k - x^*\|_2^2 \Big] \\
	&\quad = \frac{\mu}{2}(1+\gamma)\Big\|\frac{1}{1+\gamma}(z_k - x^*)+ \frac{\gamma}{1+\gamma}(x_k - x^*) - \frac{\gamma}{\mu(1+\gamma)}\nabla f_k(x_k)\Big\|_2^2 - \frac{\mu}{2}\|z_k - x^*\|_2^2  \\
	&\quad = \frac{\mu}{2}\frac{1}{1+\gamma}\Big[ \|z_k - x^*\|_2^2 + \gamma^2\|x_k - x^*\|_2^2  + \frac{\gamma^2}{\mu^2}\|\nabla f_k(x_k)\|_2^2 \\
	&\qquad + 2\gamma\langle z_k - x^*, x_k - x^*\rangle - \frac{2\gamma}{\mu}\langle z_k - x^*, \nabla f_k(x_k)\rangle \\
	&\qquad - \frac{2\gamma^2}{\mu}\langle x_k - x^*, \nabla f_k(x_k)\rangle \Big] - \frac{\mu}{2}\|z_k - x^*\|_2^2 \numberthis \label{eq:potential_change_second_kerm}.
	\end{align*}
	
	Adding \eqref{eq:potential_change_first_kerm_2} yields a bound on $\Delta\Psi_k$. Moreover, note that terms involving $\langle\nabla f_k(x_k), x_k - x^*\rangle$ and $\langle\nabla f_k(x_k), z_k - x^*\rangle$ cancel out.
	
	\begin{align*}
	&\Delta\Psi_k (1+\gamma)^{-k}\\
	&\quad\le \left(-\frac{1+\gamma}{2L} + \frac{\gamma^2}{2\mu(1+\gamma)}\right) \|\nabla f_k(x_k)\|_2^2 \\
	&\qquad + \frac{\mu\gamma}{2}\left(\frac{\gamma}{1+\gamma} - 1\right)\|x_k - x^*\|_2^2 + \frac{\mu}{2}\left(\frac{1}{1+\gamma} - 1\right)\|z_k - x^*\|_2^2 \\
	&\qquad + \frac{\mu\gamma}{1+\gamma}\langle z_k - x^*, x_k - x^*\rangle \\
	&\quad \le -\frac{\mu\gamma}{2(1+\gamma)}\big(\|x_k - x^*\|_2^2 + \|z_k - x^*\|_2^2 - 2\langle z_k - x^*, x_k - x^*\rangle\big) \\
	&\quad = -\frac{\mu\gamma}{2(1+\gamma)} \|(x_k - x^*) - (z_k - x^*)\|_2^2 \le 0,
	\end{align*}
	and the proof is complete.
\end{proof}

Now we can prove Theorem~\ref{th:nesterov_method_time_varying_function_sequence} using the potentials technique.
\begin{proof}[Proof of Theorem~\ref{th:nesterov_method_time_varying_function_sequence}]
	Following the definition of $\Psi_k$ and using the Lemma~\ref{lem:potential_change}, we obtain
	\begin{align*}
	(1 + \gamma)^N (f_N(y_N) - f^*)&\leq \Psi_N\leq \Psi_0\leq \frac{(L + \mu)R^2}{2}, \\
	f_N(y_N) - f^*&\leq \frac{(L + \mu)R^2}{2 (1 + \gamma)^N} = \frac{(L + \mu)R^2}{2}(1 - 1/\sqrt\kappa)^N.
	\end{align*}
\end{proof}

\section{Missing Proofs from Section~\ref{sec:accelerated_gossip_tw}}

\subsection{Proof of Theorem~\ref{th:chebyshev_acceleration_tw}}

Statement 1 follows directly from Algorithm~\ref{alg:accelerated_gossip_tw}. To see why statement 2 holds it is sufficient to note that $\range\mW^k = \cL^\top$.

For statement 3, denote $\ol\bx = \frac{1}{m}\one\one^\top\otimes\mI$ let us apply Theorem~\ref{th:nesterov_method_time_varying_function_sequence} to see that
\begin{align*}
\norm{\bx^T - \ol\bx}_2^2\leq 2\chi\norm{\bx^0 - \ol\bx}_2^2 (1 - 1/\sqrt\chi)^T.
\end{align*}
Taking $T = \sqrt\chi\log(4\chi)$ we obtain $\norm{\bx^T - \ol\bx}_2^2\leq 1/2\norm{\bx^0 - \ol\bx}_2^2$. If $\bx\in\cL^\top$, then $\ol\bx = 0$ and following the definition $C_T(\bx) = \bx - \bx^T$, we write
\begin{align*}
\norm{C_T(\bx)}_2 &= \norm{\bx - \bx^T}_2\leq \norm{\bx - \ol\bx}_2 + \norm{\bx^T - \ol\bx}_2 \\
&\leq (1 + 1/\sqrt{2})\norm{\bx - \ol\bx}_2, \\
\norm{C_T(\bx)}_2 &= \norm{\bx - \bx^T}_2\geq \norm{\bx - \ol\bx}_2 - \norm{\bx^T - \ol\bx}_2 \\
&\geq (1 - 1/\sqrt{2})\norm{\bx - \ol\bx}_2
\end{align*}

\subsection{Proof of Lemma~\ref{lemma:potential_decrease}}

The proof directly follows from Lemma~\ref{lem:potential_change} applied to problem~\eqref{eq:consensus_as_minimization}.

\end{document}

%% file: header.tex
\usepackage[english]{babel}
\usepackage{mathtools}
\usepackage{algorithm}
\usepackage{amsmath}
\usepackage{amssymb}
\usepackage{amsthm}
\usepackage[unicode,pdftex]{hyperref}
\usepackage[numbers]{natbib}
\usepackage{pifont}

\usepackage{booktabs}
\usepackage{tablefootnote}
\usepackage{array}
\usepackage{multirow}

\theoremstyle{plain}
\newtheorem{theorem}{Theorem}[section]

\newtheorem{lemma}[theorem]{Lemma}
\newtheorem{corollary}[theorem]{Corollary}
\theoremstyle{definition}
\newtheorem{definition}[theorem]{Definition}
\newtheorem{assumption}[theorem]{Assumption}
\theoremstyle{remark}
\newtheorem{remark}[theorem]{Remark}

\def\<#1,#2>{\langle #1,#2\rangle}
\def\({\left(}
\def\){\right)}

\def\[{\left[}
\def\]{\right]}

\DeclareMathOperator{\spn}{span}

\DeclareMathOperator{\range}{range}
\DeclareMathOperator{\swap}{swap}
\DeclareMathOperator{\NoBadVertices}{NoBadVertices}
\DeclareMathOperator{\FindCandidate}{FindCandidate}
\DeclareMathOperator{\PotentialBadVertices}{PotentialBadVertices}
\DeclareMathOperator{\InnerLoop}{InnerLoop}
\DeclareMathOperator{\Rearrange}{Rearrange}
\DeclareMathOperator{\ReverseOrders}{ReverseOrders}
\DeclareMathOperator{\AtLastLevel}{AtLastLevel}

\DeclareMathOperator{\Lap}{\mathbf{L}}
\DeclareMathOperator{\col}{col}

\DeclarePairedDelimiter\ceil{\lceil}{\rceil}

\newcommand{\N}{\mathbb{N}}
\newcommand{\R}{\mathbb{R}}

\newcommand{\mA}{{\bf A}}

\newcommand{\mI}{{\bf I}}

\newcommand{\mW}{{\bf W}}

\newcommand{\cE}{{\mathcal{E}}}

\newcommand{\cG}{{\mathcal{G}}}
\newcommand{\cH}{{\mathcal{H}}}

\newcommand{\cL}{{\mathcal{L}}}

\newcommand{\cN}{{\mathcal{N}}}

\newcommand{\cV}{{\mathcal{V}}}
\newcommand{\cW}{{\mathcal{W}}}

\newcommand{\bx}{{\bf x}}
\newcommand{\by}{{\bf y}}
\newcommand{\bz}{{\bf z}}

\newcommand{\eps}{\varepsilon}

\renewcommand{\epsilon}{\varepsilon}

\newcommand{\DP}{{\mathcal{D}\mathcal{P}}}

\newcommand{\ol}{\overline}
\newcommand{\one}{\mathbf{1}}

\newcommand{\norm}[1]{\left\| #1 \right\|}

\newcommand{\angles}[1]{\left\langle #1 \right\rangle}
\newcommand{\cbraces}[1]{\left( #1 \right)}

\newcommand{\braces}[1]{\left\{ #1 \right\}}

\newcommand{\circledOne}{\text{\ding{172}}}
\newcommand{\circledTwo}{\text{\ding{173}}}

\newcommand{\numberthis}{\addtocounter{equation}{1}\tag{\theequation}}

\usepackage[colorinlistoftodos,bordercolor=blue,backgroundcolor=blue!20,linecolor=blue,textsize=scriptsize]{todonotes}

%% file: main.bbl
\begin{thebibliography}{26}
\providecommand{\natexlab}[1]{#1}
\providecommand{\url}[1]{\texttt{#1}}
\expandafter\ifx\csname urlstyle\endcsname\relax
  \providecommand{\doi}[1]{doi: #1}\else
  \providecommand{\doi}{doi: \begingroup \urlstyle{rm}\Url}\fi

\bibitem[Bansal \& Gupta(2019)Bansal and Gupta]{bansal2019potential}
Bansal, N. and Gupta, A.
\newblock Potential-function proofs for gradient methods.
\newblock \emph{Theory of Computing}, 15\penalty0 (1):\penalty0 1--32, 2019.

\bibitem[Bazerque \& Giannakis(2009)Bazerque and
  Giannakis]{bazerque2009distributed}
Bazerque, J.~A. and Giannakis, G.~B.
\newblock Distributed spectrum sensing for cognitive radio networks by
  exploiting sparsity.
\newblock \emph{IEEE Transactions on Signal Processing}, 58\penalty0
  (3):\penalty0 1847--1862, 2009.

\bibitem[Das(2004)]{laplacian_thesis}
Das, K.
\newblock The laplacian spectrum of a graph.
\newblock \emph{Computers \& Mathematics with Applications}, 48\penalty0
  (5):\penalty0 715--724, 2004.
\newblock ISSN 0898-1221.
\newblock \doi{https://doi.org/10.1016/j.camwa.2004.05.005}.
\newblock URL
  \url{https://www.sciencedirect.com/science/article/pii/S0898122104003074}.

\bibitem[d'Aspremont et~al.(2021)d'Aspremont, Scieur, Taylor,
  et~al.]{d2021acceleration}
d'Aspremont, A., Scieur, D., Taylor, A., et~al.
\newblock Acceleration methods.
\newblock \emph{Foundations and Trends{\textregistered} in Optimization},
  5\penalty0 (1-2):\penalty0 1--245, 2021.

\bibitem[Dvinskikh \& Gasnikov(2021)Dvinskikh and
  Gasnikov]{dvinskikh2019decentralized}
Dvinskikh, D. and Gasnikov, A.
\newblock Decentralized and parallel primal and dual accelerated methods for
  stochastic convex programming problems.
\newblock \emph{Journal of Inverse and Ill-posed Problems}, 29\penalty0
  (3):\penalty0 385--405, 2021.

\bibitem[Forero et~al.(2010)Forero, Cano, and Giannakis]{forero2010consensus}
Forero, P.~A., Cano, A., and Giannakis, G.~B.
\newblock Consensus-based distributed support vector machines.
\newblock \emph{Journal of Machine Learning Research}, 11\penalty0 (5), 2010.

\bibitem[Gan et~al.(2012)Gan, Topcu, and Low]{gan2012optimal}
Gan, L., Topcu, U., and Low, S.~H.
\newblock Optimal decentralized protocol for electric vehicle charging.
\newblock \emph{IEEE Transactions on Power Systems}, 28\penalty0 (2):\penalty0
  940--951, 2012.

\bibitem[Gorbunov et~al.(2022)Gorbunov, Rogozin, Beznosikov, Dvinskikh, and
  Gasnikov]{gorbunov2022recent}
Gorbunov, E., Rogozin, A., Beznosikov, A., Dvinskikh, D., and Gasnikov, A.
\newblock Recent theoretical advances in decentralized distributed convex
  optimization.
\newblock In \emph{High-Dimensional Optimization and Probability}, pp.\
  253--325. Springer, 2022.

\bibitem[Kone{\v{c}}n{\'y} et~al.(2016)Kone{\v{c}}n{\'y}, McMahan, Yu,
  Richt{\'a}rik, Suresh, and Bacon]{konevcny2016federated}
Kone{\v{c}}n{\'y}, J., McMahan, H.~B., Yu, F.~X., Richt{\'a}rik, P., Suresh,
  A.~T., and Bacon, D.
\newblock Federated learning: Strategies for improving communication
  efficiency.
\newblock \emph{arXiv preprint arXiv:1610.05492}, 2016.

\bibitem[Kovalev et~al.(2020)Kovalev, Salim, and
  Richt{\'a}rik]{kovalev2020optimal}
Kovalev, D., Salim, A., and Richt{\'a}rik, P.
\newblock Optimal and practical algorithms for smooth and strongly convex
  decentralized optimization.
\newblock \emph{Advances in Neural Information Processing Systems}, 33, 2020.

\bibitem[Kovalev et~al.(2021{\natexlab{a}})Kovalev, Gasanov, Gasnikov, and
  Richtarik]{kovalev2021lower}
Kovalev, D., Gasanov, E., Gasnikov, A., and Richtarik, P.
\newblock Lower bounds and optimal algorithms for smooth and strongly convex
  decentralized optimization over time-varying networks.
\newblock \emph{Advances in Neural Information Processing Systems}, 34,
  2021{\natexlab{a}}.

\bibitem[Kovalev et~al.(2021{\natexlab{b}})Kovalev, Shulgin, Richt{\'a}rik,
  Rogozin, and Gasnikov]{kovalev2021adom}
Kovalev, D., Shulgin, E., Richt{\'a}rik, P., Rogozin, A., and Gasnikov, A.
\newblock Adom: Accelerated decentralized optimization method for time-varying
  networks.
\newblock \emph{arXiv preprint arXiv:2102.09234}, 2021{\natexlab{b}}.

\bibitem[Li \& Lin(2021)Li and Lin]{li2021accelerated}
Li, H. and Lin, Z.
\newblock Accelerated gradient tracking over time-varying graphs for
  decentralized optimization.
\newblock \emph{arXiv preprint arXiv:2104.02596}, 2021.

\bibitem[Li et~al.(2020)Li, Fang, Yin, and Lin]{li2020decentralized}
Li, H., Fang, C., Yin, W., and Lin, Z.
\newblock Decentralized accelerated gradient methods with increasing penalty
  parameters.
\newblock \emph{IEEE Transactions on Signal Processing}, 68:\penalty0
  4855--4870, 2020.

\bibitem[Molitierno et~al.(2000)Molitierno, Neumann, and
  Shader]{binary_tree_alg_connect}
Molitierno, J., Neumann, M., and Shader, B.
\newblock Tight bounds on the algebraic connectivity of a balanced binary tree.
\newblock \emph{ELA. The Electronic Journal of Linear Algebra [electronic
  only]}, 6, 03 2000.
\newblock \doi{10.13001/1081-3810.1040}.

\bibitem[Nedi{\'c}(2020)]{nedic2020distributed}
Nedi{\'c}, A.
\newblock Distributed gradient methods for convex machine learning problems in
  networks: Distributed optimization.
\newblock \emph{IEEE Signal Processing Magazine}, 37\penalty0 (3):\penalty0
  92--101, 2020.

\bibitem[Nedic et~al.(2017)Nedic, Olshevsky, and Shi]{nedic2017achieving}
Nedic, A., Olshevsky, A., and Shi, W.
\newblock Achieving geometric convergence for distributed optimization over
  time-varying graphs.
\newblock \emph{SIAM Journal on Optimization}, 27\penalty0 (4):\penalty0
  2597--2633, 2017.

\bibitem[Nesterov(2004)]{nesterov2004introduction}
Nesterov, Y.
\newblock \emph{Introductory Lectures on Convex Optimization: a basic course}.
\newblock Kluwer Academic Publishers, Massachusetts, 2004.

\bibitem[Rabbat \& Nowak(2004)Rabbat and Nowak]{rabbat2004distributed}
Rabbat, M. and Nowak, R.
\newblock Distributed optimization in sensor networks.
\newblock In \emph{Proceedings of the 3rd international symposium on
  Information processing in sensor networks}, pp.\  20--27, 2004.

\bibitem[Ram et~al.(2009)Ram, Veeravalli, and Nedic]{ram2009distributed}
Ram, S.~S., Veeravalli, V.~V., and Nedic, A.
\newblock Distributed non-autonomous power control through distributed convex
  optimization.
\newblock In \emph{IEEE INFOCOM 2009}, pp.\  3001--3005. IEEE, 2009.

\bibitem[Rogozin et~al.(2019)Rogozin, Uribe, Gasnikov, Malkovsky, and
  Nedi{\'c}]{rogozin2019optimal}
Rogozin, A., Uribe, C.~A., Gasnikov, A.~V., Malkovsky, N., and Nedi{\'c}, A.
\newblock Optimal distributed convex optimization on slowly time-varying
  graphs.
\newblock \emph{IEEE Transactions on Control of Network Systems}, 7\penalty0
  (2):\penalty0 829--841, 2019.

\bibitem[Rogozin et~al.(2021)Rogozin, Lukoshkin, Gasnikov, Kovalev, and
  Shulgin]{rogozin2021towards}
Rogozin, A., Lukoshkin, V., Gasnikov, A., Kovalev, D., and Shulgin, E.
\newblock Towards accelerated rates for distributed optimization over
  time-varying networks.
\newblock In \emph{International Conference on Optimization and Applications},
  pp.\  258--272. Springer, 2021.

\bibitem[Rojo \& Medina(2006)Rojo and Medina]{tight_bethe}
Rojo, O. and Medina, L.
\newblock Tight bounds on the algebraic connectivity of bethe trees.
\newblock \emph{Linear Algebra and its Applications}, 418\penalty0
  (2):\penalty0 840--853, 2006.
\newblock ISSN 0024-3795.
\newblock \doi{https://doi.org/10.1016/j.laa.2006.03.016}.
\newblock URL
  \url{https://www.sciencedirect.com/science/article/pii/S0024379506001649}.

\bibitem[Scaman et~al.(2017)Scaman, Bach, Bubeck, Lee, and
  Massouli{\'e}]{scaman2017optimal}
Scaman, K., Bach, F., Bubeck, S., Lee, Y.~T., and Massouli{\'e}, L.
\newblock Optimal algorithms for smooth and strongly convex distributed
  optimization in networks.
\newblock In \emph{Proceedings of the 34th International Conference on Machine
  Learning-Volume 70}, pp.\  3027--3036. JMLR. org, 2017.

\bibitem[Stevanović(2003)]{stevanovic_spectr_gap}
Stevanović, D.
\newblock Bounding the largest eigenvalue of trees in terms of the largest
  vertex degree.
\newblock \emph{Linear Algebra and its Applications}, 360:\penalty0 35--42,
  2003.
\newblock ISSN 0024-3795.
\newblock URL
  \url{https://www.sciencedirect.com/science/article/pii/S0024379502004421}.

\bibitem[Ye et~al.(2020)Ye, Luo, Zhou, and Zhang]{ye2020multi}
Ye, H., Luo, L., Zhou, Z., and Zhang, T.
\newblock Multi-consensus decentralized accelerated gradient descent.
\newblock \emph{arXiv preprint arXiv:2005.00797}, 2020.

\end{thebibliography}
